\documentclass[reqno, 12pt]{amsart}

\usepackage[a-2b,mathxmp]{pdfx}[2021/12/22]  
\usepackage{colorprofiles}

\usepackage{amssymb}
\usepackage{amsmath}
\usepackage{graphicx}
\usepackage{color}
\usepackage{tikz}
\usepackage{pgfplots, pgfplotstable}
\usepgfplotslibrary{groupplots}
\usepackage{enumitem}
\usepackage{hyperref}
\usepackage{subfigure}
\usepackage{caption}
\usepackage{url}
\usepackage{booktabs}
\usepackage{todonotes}
\usepackage{adjustbox}
\usepackage{mathabx}
\usepackage{pdfsync}
\definecolor{darkblue}{rgb}{0,0,.7}
\hypersetup{colorlinks=true,linkcolor=darkblue,citecolor=darkblue}

\definecolor{plch}{rgb}{1,0.5,0}

\usepackage{siunitx}
\newcommand{\numeoc}[1]{\num[round-precision=1,round-mode=places, scientific-notation=false, retain-zero-exponent=false]{#1}}

\newtheorem{theorem}{Theorem}[section]

\newtheorem{lemma}[theorem]{Lemma}
\theoremstyle{definition}

\newcommand{\nh}[1]{\left\vvvert{{#1}}\right\vvvert}
\newcommand{\nhe}[1]{\nh{{#1}}_{\eps}}
\newcommand{\nhep}[1]{\nh{{#1}}_{\eps,+}}
\newcommand*{\IV}{I_V}
\newcommand*{\IVh}{I_{\hat{V}}}
\newcommand*{\IW}{I_W}
\newcommand*{\IS}{I_\Sigma}
\newcommand*{\IQ}{I_Q}

\newcommand*{\PiR}{{\Pi^{R}}}
\newcommand*{\RF}{{R_{\scriptscriptstyle F}}}
\newcommand*{\xF}{{x_{\scriptscriptstyle F}}}
\newcommand*{\rF}{{r_{\scriptscriptstyle F}}}

\def\d{\partial}
\newcommand{\veps}{{\varepsilon}}
\newcommand{\dofs}{{\tt{dofs}}}
\newcommand{\Ureg}{U_{\operatorname{reg}}}
\newcommand{\Qreg}{Q_{\operatorname{reg}}}
\newcommand{\Sreg}{\Sigma_{\operatorname{reg}}}

\renewcommand{\div}{\operatorname{div}}
\newcommand{\id}{\operatorname{I}}
\newcommand{\hdg}{\operatorname{hdg}}
\newcommand{\curl}{\operatorname{curl}}
\newcommand{\dev}{\operatorname{dev}}
\newcommand{\eps}{\varepsilon}
\newcommand{\trans}{{\operatorname{T}}}
\newcommand{\trace}[1]{{\operatorname{tr}(#1)}}

\newcommand{\BDM}{\mathcal{BDM}^1}

\newcommand{\RT}{\mathcal{RT}^0}

\newcommand{\kk}{\mathbb{K}} 
\newcommand{\dd}{\mathbb{D}} 
\newcommand{\mm}{\mathbb{M}}
\newcommand{\skw}{\operatorname{skw}}

\newcommand{\dx}{\mathop{\mathrm{d} x}}
\newcommand{\ds}{\mathop{\mathrm{d} s}}
\newcommand{\Forall}{\text{ for all }}

\renewcommand{\hat}{\widehat}
\newcommand{\EE}{\mathcal{E}}
\newcommand{\rr}{\mathbb{R}}
\newcommand{\facets}{\mathcal{F}}
\newcommand{\mesh}{\ensuremath{\mathcal{T}}}
\newcommand{\T}{\mesh}
\newcommand{\jump}[1]{\ldbrack{{#1}}\rdbrack}
\newcommand{\mean}[1]{\{ #1 \}}
\newcommand{\GammaD}{\Gamma_D}
\newcommand{\GammaN}{\Gamma_N}

\newcommand{\qqed}{}

\begin{document}

\title[$H(\div)$-conforming velocity-vorticity approximations]
{Divergence-conforming velocity and vorticity
approximations for incompressible fluids obtained with minimal facet coupling}

\author
{J. Gopalakrishnan} 
\email{gjay@pdx.edu}
\address{Portland State University, PO Box 751, Portland OR 97207, USA} 

\author
{L. Kogler} 
\email{lukas.kogler@tuwien.ac.at}
\address{Institute for Analysis and Scientific Computing, TU Wien\\
Wiedner Hauptstra{\ss}e 8-10, 1040 Wien, Austria} 

\author
{P. L. Lederer} 
\email{p.l.lederer@utwente.nl}
\address{Department of Applied Mathematics, University of Twente,
Hallenweg 19, 7522NH  Enschede, Netherlands,} 
 
 \author{J. Sch\"oberl}
 \email{joachim.schoeberl@tuwien.ac.at}
 \address{Institute for Analysis and Scientific Computing, TU Wien\\
 Wiedner Hauptstra{\ss}e 8-10, 1040 Wien, Austria}

\begin{abstract}
  {We introduce two new lowest order methods, a mixed method, and a
hybrid Discontinuous Galerkin (HDG) method, for the approximation of
incompressible flows. Both methods use divergence-conforming linear
Brezzi-Douglas-Marini space for approximating the  velocity and the
lowest order Raviart-Thomas space for approximating  
the vorticity. Our methods are based on the physically correct viscous
stress tensor of the fluid, involving the symmetric gradient of
velocity (rather than the gradient),  provide exactly divergence-free
discrete velocity solutions, and optimal error estimates that are also
pressure robust. We explain how the methods are constructed using the
minimal number of coupling degrees of freedom per facet. The stability
analysis of both methods are based on a Korn-like inequality for
vector finite elements with continuous normal component. Numerical
examples illustrate  the theoretical findings and offer comparisons of
condition numbers between the two new methods.}

\end{abstract}
\maketitle

\section{Introduction}
\label{sec:introduction}

In this work we introduce two new methods for the discretization of
the steady incompressible Stokes equations in three space dimensions.
Let $\Omega \subset \rr^3$ be an open bounded domain with
Lipschitz boundary $\partial \Omega$ that is split into the Dirichlet
boundary $\GammaD$ and outflow boundary
$\GammaN$. The Stokes system for
the fluid {\em velocity} $u$ and the {\em pressure} $p$
is given by
\begin{subequations}\label{eq::stokes}
\begin{alignat}{2} 
  -\div(\nu \eps(u)) + \nabla p & = f && \quad \textrm{in } \Omega, \label{eq::stokes-a} \\
  \div u &=0 && \quad \textrm{in } \Omega, \label{eq::stokes-b}\\
  u &= 0 &&\quad \textrm{on } \GammaD, \label{eq::stokes-c}\\
  (-\nu\eps(u) + p I)n &= 0 &&\quad\textrm{on } \GammaN,\label{eq::stokes-d}
\end{alignat}
\end{subequations}
where $\eps(u) := (\nabla u + \nabla u ^\trans)/2$ is the symmetric
gradient, $f: \Omega \rightarrow \rr^3$ is an external body force,
$\nu$  is twice the kinematic viscosity,
$n$ is the outward unit normal vector and
$I\in\rr^{3\times 3}$ is the identity matrix.
We assume that both $\GammaD$ and $\GammaN$ have positive boundary
measure, and any rigid displacement vanishing on $\Gamma_D$
vanishes everywhere in~$\Omega$. (As usual, when $\Gamma_N$ is empty the pressure
space must be adapted to obtain a unique 
pressure~\cite{girault2012finite}, but we omit this case for simplicity.)
Next, define the {\em viscous stress tensor} \cite{LandaLifsh59}
by $\sigma = \nu \eps(u)$ and the {\em vorticity} by  $\omega =
\curl u$. Using them,  we can rewrite the above system as
\begin{subequations}
  \label{eq::mixedstressstokes}
  \begin{alignat}{2}
  \label{eq::mixedstressstokes-a}
  \nu^{-1}{\dev{\sigma}} - \nabla u + \kappa(\omega) & = 0  &&\quad  \textrm{in } \Omega,  \\
  \label{eq::mixedstressstokes-b}
  -\div  \sigma + \nabla p & = f  &&\quad  \textrm{in } \Omega,  \\
 \label{eq::mixedstressstokes-c} 
 \sigma - \sigma^\trans & = 0 && \quad \textrm{in } \Omega,  \\
  \label{eq::mixedstressstokes-d} 
  \div u &=0 &&\quad  \textrm{in } \Omega, \\
  \label{eq::mixedstressstokes-e}
  u &= 0 &&\quad \textrm{on } \GammaD, \\
  (\sigma - p I)n &= 0 &&\quad\textrm{on } \GammaN.
  \label{eq::mixedstressstokes-f}
 \end{alignat}
\end{subequations} 
Here we used the deviatoric part of the tensor $\tau$ given by
$
  \dev{\tau} := \tau - \frac{1}{3}\trace{\tau} I,
$
the matrix trace $\trace{\tau} := \sum_{i=1}^3 \tau_{ii}$, and the
operator $  \kappa: \rr^{3} \to \{\tau \in \rr^{3 \times 3}: \tau +
\tau^{\trans} = 0 \}$ defined by 
\begin{align*}
  \kappa(v) =
  \frac 1 2 
  \begin{pmatrix}
    0 & -v_3 & v_2 \\
    v_3 & 0 & -v_1 \\
    -v_2 & v_1 & 0
  \end{pmatrix}.
\end{align*}
Note the obvious identities
\begin{equation}
  \label{eq:kappa-identities}
  \nabla v = \eps(v) + \kappa(\curl v),
  \qquad
  2 \kappa(v) w = v \times w,
\end{equation}
for vector fields $v$ and $w$ (the first of which was already  used in~\eqref{eq::mixedstressstokes-a}).
We will refer to system~\eqref{eq::stokes} as the primal
formulation and system~\eqref{eq::mixedstressstokes} as the mixed formulation.

The literature on discretizations of \eqref{eq::stokes} and
\eqref{eq::mixedstressstokes} is too vast to list here.  The
relatively  recent quest for exactly divergence-free velocity
solutions and pressure-independent a~priori error estimates for
velocity, often referred to as pressure robust
estimates~\cite{linke2014role, 2016arXiv160903701L}, has rejuvenated
the field.  A recurring theme in this vast literature, from the early
non-conforming method of~\cite{CrouzRavia73} to the more
recent~\cite{LS_CMAME_2016}, is the desire to improve computational
efficiency by minimizing inter-element coupling. However, less studied
are its side effects on stability when
 the actual physical flux replaces 
the often-used simplified diffusive flux, i.e., when
\begin{align} \label{eq::nablavseps}
  -\div(\nu \eps(u))
  \quad \textrm{replaces} \quad
  -\div(\nu \nabla u),
\end{align}
even though an early work~\cite{Falk91} cautions how the lowest order
method of~\cite{CrouzRavia73} can become unstable when doing so.
Such instabilities arise because the larger null space of $\eps$
necessitates increased inter-element coupling (as explained in more
detail below) and are manifested in certain lowest order cases with
insufficient inter-element coupling.  In this work, focusing on the
lowest order case, we identify new stable finite element methods, with
the minimal necessary inter-element coupling, that yield exactly
divergence-free and pressure robust velocities.
New methods based on both the primal and the mixed
formulations are designed.

Yet another reason for focusing on the lowest order case is its
utility in preconditioning. Roughly speaking, a common strategy for
preconditioning high order Stokes discretizations involves combining
local (high order) error dampers via, say block Jacobi or other
smoothers, with a global (low order) error corrector such as multigrid
(or even a direct solver) applied  to the smaller lowest order
discretization.  From this point of view, it is desirable to have
stable low order versions (that remain stable
under~\eqref{eq::nablavseps}) of high order methods for design of
preconditioners, an interesting topic which we shall not touch upon
further in this paper.

To delve deeper into the mechanics of the above-mentioned instability,
consider the kernel of $\eps$, consisting of rigid displacements of the form
$x\to a + b \times x$ with $a,b\in\rr^3$. Reasonable methods
approximating the operator $-\div(\nu \eps(u))$ produce element
matrices whose nullspaces contain these rigid displacements. Ideally, when
these element-wise rigid displacements are subjected to the inter-element
continuity conditions of the discrete velocity space, they should
equal element-wise restrictions of a global rigid displacement on $\Omega$
(which can then be eliminated by boundary conditions).  However, if
the inter-element coupling in the discrete velocity space is so weak
to allow for the existence of a $u$ in it that does not equal a global
rigid displacement on $\Omega$ even though $u|_T$ is a rigid displacement on every
mesh element~$T$, then instabilities can arise~\cite{Falk91}.

The discrete velocity space we have in mind is the lowest order
(piecewise linear) $H(\div)$-conforming Brezzi-Douglas-Marini ($\BDM$)
space~\cite{brezzi2012mixed}.  (A basic premise of this paper is the
unquestionable utility of $H(\div)$-conforming velocity spaces to
obtain exactly divergence-free discrete Stokes velocity fields, well
established in prior works~\cite{cockburn2005locally,
  cockburn2007note, LS_CMAME_2016,mcsI, mcsII}). Hence, to understand
how to avoid the above-mentioned instability while setting velocity in  the $\BDM$
space, we ask the following question: {\em how many coupling degrees of
  freedom} (\dofs) {\em are needed to guarantee that two rigid displacements
$u_\pm$, given respectively on two adjacent elements $T_\pm$,
coincide on the common interface $F = \partial T_+ \cap \partial T_-$?}

\begin{figure}
  \centering
  \includegraphics[width=0.9\textwidth, clip = true, trim = 1cm 0 0 0 ]{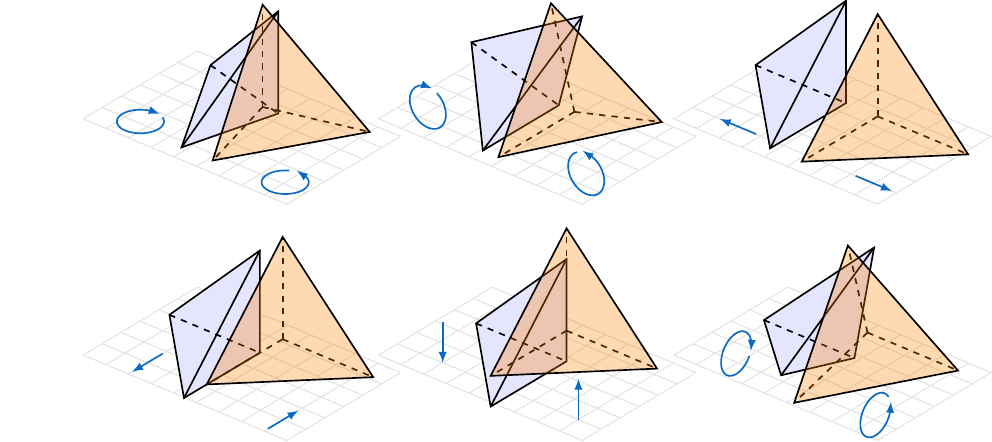}
  \caption{Configurations of adjacent elements after
    deformation by piecewise rigid displacements of two adjacent elements $T_\pm$.}
  \label{fig::motions}
\end{figure}

The pictorial representations of the deformations created by $u_\pm$
in Figure~\ref{fig::motions} lead to the answer. Three of the pictured
deformations are just translations (generated by the $a$-vector in $a
+ b \times x$).  For a unit vector $b$, letting $R^b_\theta$ denote
the unitary operator that performs a counterclockwise rotation by
angle $\theta$ around $b$, it is easy to see that $R^b_\theta x = x +
\theta (b \times x) + \mathcal{O}(\theta^2)$ as $\theta \to 0$.
Therefore the deformation created by the rigid displacement $b \times
x$ can be viewed as an infinitesimal rotation about~$b$. These
deformations are portrayed in Figure~\ref{fig::motions} as rotations
about three linearly  independent~$b$-vectors.  The first row in
Figure~\ref{fig::motions} illustrates deformations generated by
piecewise rigid displacements which are given by two $b$-vectors
coplanar with $F$ and an $a$-vector normal to $F$.  These rigid
displacements are forbidden in the $\BDM$ space. Indeed, 
recall~\cite{brezzi2012mixed} that the $\BDM$ \dofs\ on the facet
  $F$ are given by the linear functionals
  $ u \mapsto \int_F u \cdot n ~q \ds$ for all linear polynomials $q$  on $F$,
  where $n$ is a normal vector on $F$. These represent three \dofs\ illustrated in
  left diagram of Figure~\ref{fig::dofs}.
  If these three \dofs\ coincide for two
rigid displacements $u_{\pm}$, then the corresponding normal component must
be continuous on $F$. This continuity forbids the
above-mentioned deformations to be generated by 
elements of the $\BDM$ space. We summarize this by saying that the  rigid displacements
portrayed in the first row of Figure~\ref{fig::motions}
are ``controlled'' by the three $\BDM$ \dofs~of the facet $F$ which
are illustrated in the left diagram of Figure~\ref{fig::dofs}.

It remains to control the rigid displacements of the second row of
Figure~\ref{fig::motions} using three additional \dofs\ per facet.
To this end, our new methods have two additional spaces: (i) one that
approximates the in-plane components of the velocity on facets,
illustrated in the middle diagram of Figure~\ref{fig::dofs}, used to
control the first two rigid displacements in the second row of
Figure~\ref{fig::motions}; and (ii) a second space, schematically
indicated in the last diagram of Figure~\ref{fig::dofs}, that controls
the third deformation in the second row of Figure~\ref{fig::motions}.
The latter deformation arises from piecewise rigid displacements of the form
$u_\pm = b_\pm \times x$ with $b_\pm$ collinear to $n$, a unit normal
of $F$.  Since $\curl(b_\pm \times x) = 2 b_\pm$, we can make the two
rigid displacements coincide on $F$ by requiring continuity of
$n \cdot \curl u_\pm$. While continuity of $n \cdot \curl u$
certainly holds if $u$ is the exact Stokes velocity, it does not
generally hold for $u$ in $\BDM$. Hence, keeping in view that
$\omega = \curl u$ represents vorticity, {\em we incorporate this
  constraint in our new methods by approximating vorticity $\omega$ in the lowest
  order Raviart-Thomas space.} This is our second additional space.
Its single {\tt dof}  per facet  is shown schematically in the last diagram of
Figure~\ref{fig::dofs}.

\begin{figure}
  \centering
  \includegraphics[width=0.9\textwidth, clip = true, trim = 1cm 0 0 0 ]{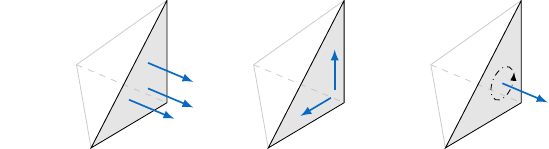}
  \caption{Classification of facet \dofs\ in our new methods into three types:
    (1) normal velocity components in the form of  $\BDM$ facet dofs,
    (2) tangential facet velocities,
    (3) normal vorticity as $\RT$ facet dof. 
  }
  \label{fig::dofs}
\end{figure}

In the first part of the paper we will employ these additional spaces
to construct a novel HDG method to approximate~\eqref{eq::stokes} and
present a detailed stability and error analysis.
HDG methods have become popular ever since its introduction in~\cite{cockburn2009unified} which showed  how interface variables,
or facet variables, 
can be effectively used to construct DG schemes amenable to
static condensation.
In the method presented here, the interface variable approximates the
tangential components of the velocity.
The key technical
ingredient in the analysis that reflects the insight garnered from the
above pictorial discussion is a discrete Korn-like inequality for the
$\BDM$ space
(see Lemma~\ref{lemma::korn} below, 
and the version with interface variables  in
Lemma~\ref{lem::korn_hdg}).

The second part of this work discusses the derivation of a novel mixed
method for the approximation of \eqref{eq::mixedstressstokes} and is
motivated by our previous two papers~\cite{mcsI,mcsII} and the many
other works on discretizing~\eqref{eq::mixedstressstokes} such as
\cite{Farhloulcanadian, MR1231323, MR1464150, MR1934446}. In
\cite{mcsII} we derived the ``\underline{M}ass-\underline{C}onserving \underline{S}tress-yielding'' (MCS)
formulation where the symmetry of $\sigma$ was incorporated in a weak sense
by means of a Lagrange multiplier that approximates
$\omega = \curl u$. While the $\omega$ there was approximated using
element-wise linear (or higher degree) functions {\em without} any
inter-element continuity requirements, the new mixed method we propose
here will approximate $\omega$ in the lowest order Raviart-Thomas
space instead.  The lowest order case that was proved to be stable
in~\cite{mcsII} had nine coupling \dofs\ per facet. We are able
to reduce this number to the minimal six (the dimension of rigid
displacements) in this paper.
This minimal coupling was achieved earlier in~\cite{TaiWinth06} using
a bubble-augmented velocity space which is a subspace of a degree-four
vector polynomial space. Since higher degrees necessitate more
expensive integration rules, we offer our simpler elements as an
alternative.

Other methods that approximate the operator $\div(\nu \nabla u)$,
such as~\cite{CrouzRavia73,LS_CMAME_2016, mcsI}, are able to reduce
the number of coupling \dofs\ per facet even further. Since our focus
here is on methods that approximate $\div(\nu \eps( u))$, we restrict
ourselves to a brief remark on this.  Since the kernel of $\nabla$ (applied to vector fields) is
three dimensional, we expect the minimal number of coupling \dofs\ per
facet to be three when approximating $\div(\nu \nabla u)$. A method
with this minimal coupling was
achieved early by~\cite{CrouzRavia73}.  To also obtain pressure robust and
exactly divergence-free solutions, prior works~\cite{LS_CMAME_2016,
  mcsI} settled for a slightly higher five coupling \dofs\ per facet
in the lowest order case.  It is now known that this can be improved
by employing the technique of ``relaxed
$H(\div, \Omega)$-conformity,''
see~\cite{ledlehrschoe2017relaxedpartI,ledlehrschoe2018relaxedpartII},
which results in a method with the minimal three coupling \dofs\ per facet and yet,
thanks to a simple post-processing, provides optimal convergence
orders and pressure robustness.
While on the subject of coupling \dofs, an 
  explanation of our focus on three-dimensional (3D) domains is in order.
  On two-dimensional (2D)
domains, the space of rigid displacements only has three
dimensions. In the lowest order 2D case, $\BDM$ space provides two coupling
\dofs\ per facet edge, and the space of tangential facet velocities
adds one more coupling degree of freedom. Thus the minimal facet
coupling (of three \dofs) needed to eliminate the rigid
displacements are more immediate in 2D case when compared to the 3D
case, which is why restrict to the 3D case henceforth.

The new HDG method and the new mixed method proposed in this paper
both have the same coupling \dofs, the same velocity convergence orders
and the same structure preservation properties like pressure robustness
and mass conservation. On closer comparison, two advantages of the
mixed method are notable. One is its direct approximation of viscous
stresses. Another is the absence of any stabilization parameters in
it. In fact, in our numerical studies, the conditioning of a matrix
block arising from the parameter-free mixed method was found to be better
than the analogous HDG block for all ranges of the HDG stabilization
parameter we considered.

{\bf{Outline.}} We set up general notation in
Section~\ref{sec::prelim} and continue with a description of the
variational framework used throughout the paper.  Finite element spaces, a discrete
Korn-like inequality, and resulting norm equivalences are introduced
in Section~\ref{sec::fem_ne_io}. A list of interpolation operators
into these spaces and their properties with references to literature
can also be found there. In Section~\ref{sec::hdg} we introduce and
analyze the HDG method for the primal set of equations
\eqref{eq::stokes} and in Section~\ref{sec::mcs} we do the same for
the MCS method for the mixed set of equations
\eqref{eq::mixedstressstokes}. Finally, in Section \ref{sec::numex} we
perform numerical experiments to illustrate and complement our
theoretical findings.

\section{Notation and weak forms} \label{sec::prelim}

 By $\mm$ we denote the vector space of real $3 \times 3$ matrices and
by $\kk$ we denote the vector space of $3 \times 3$ skew symmetric
matrices, i.e., $\kk = \skw(\mm)$, where $\skw \tau
= \frac{1}{2} (\tau - \tau^T)$ for $\tau \in \mm$. Further, let
$\dd = \dev{(\mm)}$. To indicate vector and matrix-valued functions on
$\Omega$, we include the range in the notation, thus while
$L^2(\Omega) = L^2(\Omega, \rr)$ denotes the space of square
integrable and weakly differentiable $\rr$-valued functions on
$\Omega$, the corresponding vector and matrix-valued function spaces
are defined by
$L^2(\Omega, \rr^3) := \left\{ u : \Omega \to \rr^3 \big| \;u_i \in
L^2(\Omega)\right\}$ and $  L^2(\Omega,\mm ) := \left\{ \tau : \Omega
\to {\mm} \big|\; \tau_{ij} \in L^2(\Omega)\right\},$ respectively.
For any ${\tilde{\Omega}} \subseteq \Omega$,  we
denote by $(\cdot,\cdot)_{\tilde{\Omega}}$ the inner product on $L^2({\tilde{\Omega}})$
(or its vector- or matrix-valued versions). Similarly, we extend this
notation and write $\| \cdot \|_{{\tilde{\Omega}}}$ for the corresponding
$L^2$-norm of a (scalar, vector, or matrix-valued) function on the domain ${\tilde{\Omega}}$. In the case ${\tilde{\Omega}} =
\Omega$ we will omit the subscript in the inner product, i.e. we have
$(\cdot,\cdot)_{\tilde{\Omega}} = (\cdot,\cdot)$ and we will use the notation
$\| \cdot \|_0 = \| \cdot \|_\Omega$.

 In addition to the differential operators we have already used, 
 $\nabla, \eps, \curl$, 
we understand $\div \Phi $ as either
$\sum_{i=1}^3 \partial_i \Phi_i$ for a vector-valued function $\Phi$,
or the row-wise divergence $\sum_{j=1}^3 \partial_j \tau_{ij}$ for a
matrix-valued function $\tau$. 
In addition to the standard
Sobolev spaces $H^m(\Omega)$ for any $m\in \rr$, we shall also use the
well-known spaces
$H(\div,\Omega) = \{ v \in L^2(\Omega, \rr^3):
 \div v  \in L^2(\Omega) \}$
and $ H(\curl,\Omega) = \{ v \in L^2(\Omega, \rr^3):
 \curl v \in L^2(\Omega,\rr^3) \}.$
 We use $H^1_{0,B}(\Omega)$, $H_{0,B}(\div, \Omega)$ and
 $H_{0,B}(\curl, \Omega),$ to denote the spaces of functions whose
 trace, normal trace and tangential trace respectively vanish on
 $\Gamma_B$, for $B \in \{ D, N \}$.
The only somewhat nonstandard Sobolev space that we shall use is 
\begin{align}
  \label{eq:Hcurldiv}
  H(\curl \div, \Omega) := \{ \tau \in L^2(\Omega,\dd): \div \tau  \in H_{0,D}(\div, \Omega)^*\},
\end{align}
where $H_{0,D}(\div, \Omega)^*$ is the dual space of
$H_{0, D}(\div, \Omega)$. In the case $\Gamma_D = \partial\Omega$, as
proved in \cite{mcsI}, the dual of $H_{0, D}(\div, \Omega)$ equals
$H^{-1}(\curl, \Omega)$, so the condition that
$\div \tau  \in H_{0,D}(\div, \Omega)^*$ in \eqref{eq:Hcurldiv}
is the same as requiring that $\curl \div \tau \in H^{-1}(\Omega)$.
This explains the presence of the operator ``$\curl \div$'' in the
name of the space in~\eqref{eq:Hcurldiv}.

We denote by $\mesh$ a quasiuniform and shape regular triangulation of
the domain $\Omega$ into tetrahedra. Let $h$ denote the maximum of the
diameters of all elements in $\mesh$. Throughout this work we write $A
\sim B$ when there exist two constants $c,C >0$ {\em independent of
the mesh size $h$ as well as the viscosity $\nu$} such that $cA \le B
\le C A$. Similarly, we use the notation $A \lesssim B$ if there
exists a similar constant $C$  (independent of $h$ and $\nu$) such
that $A \le CB$. Henceforth we assume that $\nu$ is a constant. Due to
quasiuniformity we have $h \sim \textrm{diam}(T)$ for any $ T\in
\mesh$. The set of element interfaces and boundaries is denoted by
$\facets$. 
{This set is further split into facets on the
  Dirichlet  boundary,
  $\facets_D = \{ F \in \facets: F \subset \Gamma_D\}$,
  facets on the  Neumann boundary $
  \facets_N = \{ F \in \facets : F \subset \Gamma_N\}$ and
  facets in the interior $\facets^{0} = \facets \setminus ( \facets_N  \cup \facets_D)$.
  Also let $\facets_{0, D}  = \facets_0 \cup \facets_D$.}

For piecewise smooth functions $v$ on the mesh, $\jump{v}$ and
$\mean{v}$ are functions on $\facets$ whose values on each interior
facet equal the jump (defined up to a sign) of $v$ and the mean of the
values of $v$ from adjacent elements. On boundary facets, they are
both  defined to be the trace of $v$. On each element boundary, and
similarly on each facet on the global boundary we denote by $n$ the
outward unit normal vector. Then the normal and tangential trace of a
smooth enough  vector field $v$ is given by 
\begin{align*}
 v_n = v\cdot n \quad \textrm{and} \quad v_t = v - v_n n.
\end{align*}
Accordingly, the normal trace is a scalar function and
the tangential trace is a vector function. In a similar manner we
introduce the normal-normal ($nn$) trace
and the normal-tangential $(nt)$ trace of a
matrix valued function $\tau$ by
\begin{align*}
  \tau_{nn} := \tau : n \otimes n  = n^\trans \tau n \quad \textrm{and} \quad \tau_{nt} = (\tau n)_t.
\end{align*} 
For any $\tilde{\Omega} \subseteq \Omega$, we denote by
$P^k(\tilde{\Omega}) = P^k(\tilde{\Omega},\rr) $ the set of polynomials of degree at
most $k$, restricted to $\tilde{\Omega}$.
Let $P^k(\tilde{\Omega},\rr^3)$ and $P^k(\tilde{\Omega},\mm)$ denote the
analogous vector- and matrix-valued versions whose components are in
$P^k(\tilde{\Omega})$. With respect to these spaces we then define
$\Pi^k_{\tilde{\Omega}}$, the $L^2(\tilde{\Omega})$-projection into the space
$P^k(\tilde{\Omega})$ or its vector- or matrix-valued versions.
We omit
subscript from $\Pi^k_{\tilde{\Omega}}$  if it is clear from 
context. For the space of functions the restrictions of which are in
$P^k(T)$ for all $T\in\mesh$ we write simply $P^k(\mesh)$. The
analogous convention holds for $H^k(\mesh), L^2(\facets)$, etc. 

The standard \cite{girault2012finite} variational formulation of
\eqref{eq::stokes} is to find
$(u,p)\in H^1_{0,D}(\Omega, \rr^3)\times L^2(\Omega)$ such that
\begin{subequations} \label{eq::stokesweak}
  \begin{alignat}{2}
    \nu (\eps(u),\eps(v)) - (\div v,p) &= (f,v) && \quad\Forall v
    \in H^1_{0,D}(\Omega, \rr^3), \label{eq::honediffusive} \\
     - (\div u,q) &= 0 && \quad\Forall q \in L^2(\Omega).
     \label{eq::incomp}
  \end{alignat}
\end{subequations}
However our novel methods use 
$H(\div)$-conforming spaces for the
approximation of the velocity $u$. 
Another weak form where velocity is set in  $H(\div)$ was given
in~\cite{mcsI,mcsII,lederer2019mass} using 
$
\Sigma^{\operatorname{sym}} := \{ \tau \in H(\curl \div, \Omega):  \;\tau = \tau^\trans\}.
$
It finds $(\sigma, u, p) \in
\Sigma^{\operatorname{sym}} \times H_{0,D}(\div, \Omega) \times L^2(\Omega)$
such that
\begin{subequations} \label{eq::mixedstressstokesweak}
  \begin{alignat}{2}
    (\nu^{-1} \sigma, \tau) + \langle \div \tau,  u\rangle_{\div} &
    = 0 &&\quad\Forall \tau \in \Sigma^{\operatorname{sym}}, \\
    -\langle \div \sigma ,  v\rangle_{\div} - (\div v, p) & = (f, v)
    &&\quad\Forall v \in H_{0,D}(\div, \Omega),    
    \\
    -(\div u, q) &=0 
    &&\quad\Forall q \in L^2(\Omega),
\end{alignat}                     
\end{subequations}
where
$
\Sigma^{\operatorname{sym}} := \{ \tau \in H(\curl \div, \Omega):  \;\tau = \tau^\trans\}.
$
Here $\langle \cdot, \cdot \rangle_{\div}$ denotes the duality pairing on
$H_{0,D}(\div, \Omega)^* \times H_{0,D}(\div, \Omega)$. Note that
since $\sigma \in L^2(\Omega, \dd)$ we have $\trace{\sigma} = 0$ which
is motivated by \eqref{eq::mixedstressstokes-a}. In
\cite{lederer2019mass}, a detailed well-posedness analysis of
\eqref{eq::mixedstressstokesweak} was provided, but in this paper,  \eqref{eq::mixedstressstokesweak} will serve merely to motivate
the new mixed method of Section~\ref{sec::mcs}.

\section{The finite elements used and their properties}\label{sec::fem_ne_io}

In this preparatory section, we define the standard finite
  element spaces used to construct our methods, their natural
  interpolators, and a number of discrete norm equivalences revealing
  equivalent norms involving piecewise
  $\veps(\cdot)$. Lemma~\ref{lem::korn_hdg} below will be used in the
  analysis of the HDG scheme while the analysis of the MCS scheme will additionally need 
  Lemmas~\ref{lemma::normequi_tech}--\ref{lemma::normequi}.
  We begin with the finite element spaces used in this paper:
\begin{subequations} \label{def::fespaces}
  \begin{align}
    V_h &:= \{v_h \in H_{0,D}(\div,\Omega): v_h|_T \in P^1(T,\rr^3)\}, \label{def::fespaces_V} \\\
    \hat V_h &:= \{ \hat v_h \in L^2(\facets, \rr^3):
               \hat v_h = 0 \textrm{ on } \Gamma_D, \text{ and for all }
               F \in \facets,
    \nonumber \\
    & \hspace{3.85cm} \hat v_h|_F \in P^0(F,\rr^3) \text{ and }  (\hat v_h)_n|_F = 0  \}, \label{def::fespaces_Vhat}\\
    W_h &:= \{ \eta_h \in H_{0,D}(\div, \Omega): \eta_h|_T \in P^0(T, \rr^3) + x  P^0(T,\rr)\Forall T \in \mesh \},\label{def::fespaces_W}\\
    \Sigma_h &:= \{ \tau_h \in L^2(\Omega, \dd): \tau_h|_T \in P^1(T, \dd), (\tau_h)_{nt}|_F \in P^0(F, n_F^\perp) \}, \label{def::fespaces_Sigma}\\
    Q_h &:= P^0(\mesh).\label{def::fespaces_Q}
  \end{align}
\end{subequations}
Note that for any $\tau_h \in \Sigma_h$, on a facet $F$, 
  $(\tau_h)_{nt}$ is a constant
  function on $F$ taking values in $n_F^\perp$, where $n_F^\perp$ denotes the
  orthogonal complement of $n_F$, a unit normal of $F$. This is indicated by the
  notation  $(\tau_h)_{nt} \in P^0(F,n_F^\perp)$  in~\eqref{def::fespaces_Sigma}.
Also any $\hat v_h \in \hat V_h$ is tangential and 
takes values in $n_F^\perp$
on each facet~$F$.
Note also that $V_h$, which equals $H_{0,D}(\div,\Omega) \cap \BDM$ in
the notation of \S\ref{sec:introduction}, is the lowest order
Brezzi-Douglas-Marini space while $W_h$ is the lowest order
Raviart-Thomas space~\cite{brezzi2012mixed}.
The space $\Sigma_h$ is a discontinuous version of the ``$nt$-continuous''
space introduced in~\cite{mcsI},
for which simple shape functions were exhibited there.
All of these finite element spaces are obtained by mappings from a
single reference finite element.  (All these maps extend to curvilinear
elements, although we restrict to affine equivalent elements in our
analysis here.)  The maps are compatible with the degrees of freedom of the
spaces. (For $\Sigma_h$, the appropriate map is given in \cite{mcsI} and
compatibility with degrees of freedom is proved in
\cite[Lemma~5.7]{mcsI}, while for the other spaces, the mappings are
standard.)  In the case of $V_h$ and $W_h$, the maps are Piola maps
which also preserve divergence-free subspaces.

\subsection{A discrete Korn-type inequality}
\label{ssec:discrete-korn-type}

Korn inequalities for piecewise functions were given in
\cite[Theorem~3.1]{brenner_korn}. A further refinement  was given
in \cite[Theorem~3.1]{MardaWinth06}. To describe it, let $\PiR$ denote
the facet-wise $L^2$ projection onto
$\RF := \{t + \alpha~ n\times x : t \in n^\perp,~\alpha\in\rr\}$, the
space of tangential components (on a facet~$F$) of the rigid
displacements (or simply the space of two-dimensional rigid
displacements on $F$).  Let
$ H^1_{n, D}(\mesh,\rr^3) := \{ u: u \in H^1(T,\rr^3)$ for all elements
$T \in \mesh$ and $ {\jump{u}}_n = 0$ on all facets {$F \in \facets_{0, D} \}.$} A minor modification of the proof of
\cite[Theorem~3.1]{MardaWinth06} shows that 
\begin{align}\label{eq::korn_c}    
  \| \nabla u \|_\mesh^2
  & \;\lesssim\;
    \| \eps(u) \|_\mesh^2+  h^{-1}
    \big\| \PiR\jump{u}_t \big\|_{{\facets_{0, D}}}^2
    \quad\text{ for all } 
    u \in
    H^1_{n, D}(\mesh, \rr^3).
\end{align}
Here and throughout, we use $\| \cdot \|_\mesh^2$
to abbreviate
$\sum_{T \in \mesh} \| \cdot \|_T^2$  with the
understanding that any derivative operators in the argument of these
norms are evaluated summand by summand, e.g., the gradient and
$\varepsilon$ are evaluated element by element in~\eqref{eq::korn_c}.
This notation is similarly extended to facets, so 
$\| \cdot \|_{{\facets_{0, D}}}^2
=\sum_{F \in \facets_{0, D}} \| \cdot \|_F^2$.
Note how normal components are controlled in~\eqref{eq::korn_c}
through the space $H^1_{n, D}(\T, \rr^3),$ while tangential components
are controlled through the jumps $\jump{u}_t$.
The next result shows that
a part of the right hand side of~\eqref{eq::korn_c}
can be traded for 
a norm of  the jump of $n \cdot \curl u$ when $u$ is
in~$V_h$.

\begin{lemma}\label{lemma::korn}
  For all $u_h \in V_h$,
  \begin{align}
    \begin{split}\label{eq::korn_d}
       \|\eps(u_h)\|_\mesh^2
       +
       h^{-1}\big\|\PiR\jump{u_h}_t\big\|_{{\facets_{0, D}}}^2
       \sim\;
       \|\eps(u_h)\|_\mesh^2
       & + 
        h^{-1}\big\|\Pi^0\jump{u_h}_t\big\|_{{\facets_{0, D}}}^2
        \\
        & 
      + h\big\|\jump{\curl u_h}_n\big\|_{{\facets_{0, D}}}^2.
    \end{split}
  \end{align}
\end{lemma}

\begin{proof}
  By Pythagoras theorem, 
  \[\big\|\PiR \jump{u_h}_t\big\|^2_F =
  \big\|\Pi^0\jump{u_h}_t\big\|^2_F +
  \big\|(\PiR-\Pi^0)\jump{u_h}_t\big\|^2_F.\]
  Hence  \eqref{eq::korn_d} would follow
  once we prove that for all $F \in \facets$ and all $u_h \in V_h$, 
  \begin{equation}\label{eq::korn_ts}
    \begin{split}
    h\big\|\jump{\curl u_h}_n&\big\|_F^2
    + \sum_{T \in \T_F}\big\|\eps(u_h)\big\|^2_T 
    \;\sim\;
    h^{-1}\big\|(\PiR-\Pi^0)\jump{u_h}_t\big\|_F^2
    + \sum_{T\in \T_F}\big\|\eps(u_h)\big\|^2_T
    \end{split},
  \end{equation}
  where $\T_F = \{ T \in \T: F \subset \partial T\}$.

  To
  prove~\eqref{eq::korn_ts}, first note that, restricted to every
  facet~$F$, $\PiR-\Pi^0$ is the $L^2(F)$-orthogonal projection onto
  the one dimensional span of $r_F = n_F \times (x - x_F)$ where
  $\xF=\frac{1}{|F|}\int_Fx\dx$ is the barycenter of
  $F$.
  Computing this one-dimensional projection,
  $(\PiR-\Pi^0)\jump{u_h}_t\big|_{F}
  = (\rF, \jump{u_h})_F \,r_F/\|\rF\|_F^2.$
  Therefore,
  \begin{equation}
    \label{eq::korn_pis_equality}
    \big\|(\PiR-\Pi^0)\jump{u_h}_t\big\|_F
    = \frac{|(\rF, \jump{u_h})_F|}{\|\rF\|_F}.
  \end{equation}
  
  To  simplify the numerator of the last term,
  let $w$ equal $u_h|_T$ for some $T \in \T_F$. We claim
  that
  \begin{equation}
    \label{eq:ident_rF}
    (r_F, w)_F
    = (r_F, \veps(w) (x - x_F))_F + \frac 1 2 (|x - x_F|^2, n_F \cdot \curl w)_F.
  \end{equation}
  To see why, recalling that $w$ is linear in $T$ (and hence in $F \subset \d T$), for any $x \in F$, 
  \begin{align*}
    w(x) & = w(x_F) + \nabla w\, ( x- x_F)
    \\
    & = w(x_F) + \veps(w) (x- x_F) + \frac 1 2 \curl w \times (x-x_F),
  \end{align*}
  where we have used~\eqref{eq:kappa-identities}. Since $r_F$ is
  orthogonal to constants on $F$,
  \[
    (r_F, w) =
    (r_F, \veps(w) (x - x_F))_F + \frac 1 2
    (n_F \times (x - x_F),  \curl w \times (x - x_F))_F.
  \]
  Now, since $(x-\xF)\perp n$ for any $x \in F$, using the identity
  $(a\times b) \cdot (c\times b) = |b|^2(a\cdot c) - (a\cdot b)(c\cdot
  b)$ to simplify the  last term, we obtain~\eqref{eq:ident_rF}.

  The equivalence of~\eqref{eq::korn_ts} is a consequence of  the identity 
  \begin{equation}
    \label{eq:ident_Pi_eps_curl}
    \big\|(\PiR-\Pi^0)\jump{u_h}_t\big\|_F
    = \frac{(\rF, \jump{\veps(u_h)} (x - x_F))_F}{\|\rF\|_F}
    + \frac{( |x - x_F|^2, \jump{\curl u_h}_n )_F}{2\|\rF\|_F},    
  \end{equation}
  immediately obtained by combining~\eqref{eq::korn_pis_equality}
  and~\eqref{eq:ident_rF}. Indeed, by applying Cauchy-Schwarz
  inequality to the terms on the right hand side
  of~\eqref{eq:ident_Pi_eps_curl}, simple local scaling arguments give
  $h^{-1}\big\|(\PiR-\Pi^0)\jump{u_h}_t\big\|_F^2 \lesssim \sum_{T \in
    \T_F} \| \veps(u_h) \|_T^2 + h\big\|\jump{\curl u_h}_n\big\|_F^2,
  $ thus proving one side of equivalence in~\eqref{eq::korn_ts}. To prove
  the other side, we begin by 
  noting that $\curl (u_h)$ is constant on each element, so
  \begin{align*}
    h^{1/2} \big\|\jump{\curl u_h}_n\big\|_F
      & \lesssim  h^{-1/2}
        \frac{( |x - x_F|^2, \jump{\curl u_h}_n )_F}{2\|\rF\|_F}
    \\
      & = h^{-1/2} \left(
        \big\|(\PiR-\Pi^0)\jump{u_h}_t\big\|_F
        - \frac{(\rF, \jump{\veps(u_h)} (x - x_F))_F}{\|\rF\|_F}
        \right)
    \\
      & \lesssim
        h^{-1/2} \left\|(\PiR-\Pi^0)\jump{u_h}_t\right\|_F
        +\sum_{T \in \T_F}\|\eps(u_h)\|_T,
  \end{align*}
  where we have used~\eqref{eq:ident_Pi_eps_curl} and local scaling
  arguments again.  Squaring both sides and applying Young's
  inequality, \eqref{eq::korn_ts} is
  proved.
  \qqed
\end{proof}

\subsection{Norm equivalences}

The product space for the kinematic variables is given by $U_h
:= V_h \times \hat V_h \times W_h$. For the analysis we define the  norms 
\begin{align*}
  \| u_h, \hat u_h \|^2_{\nabla}
  &:=
    \| \nabla u_h \|_\mesh^2 + h^{-1} \| \Pi^0(u_h - \hat u_h)_t\|^2_{\partial \T},
  \\
  \nhe{u_h, \hat u_h, \omega_h }^2
  &:= \| \eps(u_h) \|_\mesh^2 + h^{-1} \| \Pi^0 (u_h - \hat u_h)_t\|^2_{\partial \T} + h \| (\curl u_h - \omega_h)_n\|_{\partial \T}^2, \\    
  \| u_h, \hat u_h, \omega_h \|^2_{\eps}
  &:= \| \eps(u_h) \|_\T^2  + h^{-1} \| \Pi^0(u_h - \hat u_h)_t\|^2_{\partial \T} + \| \curl u_h - \omega_h\|_\T^2, \\
  \| u_h, \hat u_h, \omega_h \|^2_{U_h}
  &:= \| \dev \nabla u_h - \Pi^0 \kappa(\omega_h) \|_\T^2 + h^{-1} \| \Pi^0 (u_h - \hat u_h)_t\|^2_{\partial \T}, 
\end{align*}
where we have used $\| \cdot \|_{\d \T}^2$ to abbreviate
$\sum_{T \in \mesh} \| \cdot \|_{\d T}^2$.  We will shortly establish
relationships between these norms (Lemma~\ref{lemma::normequi}).  That
these are all norms on $U_h$ may not be immediately obvious, but follows from
Lemma~\ref{lem::korn_hdg} below (where we critically use that $\hat u_h$ is
single valued on facets).  As we shall see later, $ \nhe{\cdot}$ is
the natural norm to analyze the new HDG method in \S\ref{sec::hdg},
while $\| \cdot \|_\veps$ features in the analysis of the MCS method
in \S\ref{sec::mcs}. All the above norms involve the interface
variable $\hat u_h,$ so they may be referred to as ``HDG-type''
norms. In contrast, ``DG-type'' norms were used in
Subsection~\ref{ssec:discrete-korn-type}, where
Lemma~\ref{lemma::korn} and~\eqref{eq::korn_c} imply
\begin{equation}
  \label{eq:korn-inter}
  \| \nabla u_h \|_{\mesh} 
  \;\lesssim\;        \|\eps(u_h)\|_\mesh^2 + 
      h^{-1}\big\|\Pi^0\jump{u_h}_t\big\|_{\facets_{0, D}}^2
      + h\big\|\jump{\curl u_h}_n\big\|_{\facets_{0, D}}^2.  
\end{equation}
A similar discrete Korn-type inequality also holds for HDG-type norms,
as seen in the next lemma. 

\begin{lemma}\label{lem::korn_hdg}
  For all $(u_h, \hat{u}_h, \omega_h)\in U_h$, we have  the Korn-like inequality
  \begin{align}\label{eq::korn_hdg}
    \| u_h, \hat{u}_h \|_{\nabla} \lesssim
    \nhe{u_h, \hat{u}_h, \omega_h}.
  \end{align}
  The reverse inequality holds in the sense that
  for any $(u_h, \hat{u}_h)\in V_h\times\hat{V}_h$
  there exists a $\omega_h\in W_h$ such that
  \begin{align}\label{eq::korn_hdg_reverse}
    \nhe{u_h, \hat{u}_h, \omega_h} \lesssim \| u_h, \hat{u}_h \|_{\nabla}.
  \end{align}
\end{lemma}
\begin{proof}
  To prove~\eqref{eq::korn_hdg},
  first note that
  on an interior facet $F = \d T_+ \cap \d T_- \in \facets_0$ shared
  by two elements $T_\pm \in \T$, letting $u_h^\pm = u_h|_{T_\pm}$,
  since $\hat u_h$ is single valued on $F$, we have
  $ u_h^+ - u_h^- = (u_h^+ - \hat u) - (u_h^- - \hat u).$ Moreover, on
  a facet $F \in \facets_D$, $ u_h|_F = (u_h - \hat u_h)|_F$.  Thus by
  triangle inequality,
  \begin{align}
    \label{eq:100}
    \big\|\Pi^0\jump{u_h}_t\big\|_{\facets_{0, D}}^2
    & \le 2\|\Pi^0(u_h - \hat{u}_h)_t\|_{\partial \T}^2,
  \end{align}
  where we have increased the right hand side to include facets on
  $\Gamma^N$ also.  Similarly, since the normal component of the given
  $\omega_h \in W_h$ is continuous across $F \in \facets_0$ and zero on
  $F \in \facets_D$,
  \begin{align}
    \label{eq:101}
    \big\|\jump{\curl u_h}_n\big\|_\facets^2
    \le 2 \left\|(\curl u_h - \omega_h)_n\right\|_{\partial \T}^2.
  \end{align}
  Using~\eqref{eq:100} and~\eqref{eq:101} in~\eqref{eq:korn-inter},
  we obtain the estimate~\eqref{eq::korn_hdg}.
  
  To prove~\eqref{eq::korn_hdg_reverse}, consider a  function
  $\omega_h\in W_h$ satisfying
  \[
    \begin{aligned}
      n \cdot \omega_h & =   n \cdot \{ \curl u_h\}
      && \text{ on } \d T \setminus \Gamma^D, 
      \\
      n \cdot \omega_h & = 0
      && \text{ on } \d T \cap  \Gamma^D, 
    \end{aligned}
  \]
  on the boundary of every element $T \in \T$. Since $\curl u_h$ is
  piecewise constant, by the well known degrees of freedom of the
  Raviart-Thomas space $W_h$, these conditions uniquely fix an
  $\omega_h \in W_h$.
  Then,
  $ \| (\curl u_h - \omega_h)_n\|_F$ equals zero for
  $F \in \facets_N$, equals $\frac 1 2 \big\| \jump{\curl u_h}_n\big\|_F $ for
  $F \in \facets_0$, and equals $\| (\curl u_h)_n \|_F $ for
  $F \in \facets_D$, so 
  \[
    \| (\curl u_h - \omega_h)_n \|_{ \d \T}^2
    \lesssim
    \big\| \jump{\curl u_h}_n \big\|_{\facets_{0, D}}^2.
  \]
  Therefore, for this choice of $\omega_h$, we have 
  \begin{align*}
    \nhe{ u_h, \hat u_h, \omega_h }^2
    & \lesssim
      \| \eps(u_h) \|_\mesh^2
      + h \big\| \jump{\curl u_h}_n\big\|_{\facets_{0, D}}^2
      + h^{-1} \| \Pi^0 (u_h - \hat u_h)_t\|^2_{\partial \T}.
  \end{align*}
  By a local scaling argument
  $ h \big\| \jump{\curl u_h}_n\big\|_{\facets}^2 \lesssim \|
  \curl u_h \|_\T^2.$ Using this in the above inequality and
  recalling that
  $\| \nabla u_h \|_\T^2 = \| \veps(u_h) \|_\T^2 + \| \kappa( \curl
  u_h) \|_\T^2$, we complete the proof of~\eqref{eq::korn_hdg_reverse}.
  \qqed
\end{proof}

\begin{lemma}\label{lemma::normequi_tech} 
  For any $u_h\in V_h$, $\omega_h\in W_h$, and  $T \in\mesh$,
    \begin{subequations} \label{eq::normequi}
    \begin{align}
      \| \curl u_h - \omega_h\|_T^2 \sim&~
                                           h \| (\curl u_h - \omega_h)_n\|_{\partial T}^2, \label{eq::normequiA}\\
      \| \curl \kappa (\omega_h) \|_T \sim&~ | \kappa (\omega_h) |_{H^1(T)}^2 \sim \| \div \omega_h \|_T,
                                              \label{eq::normequiB} \\
      \begin{split}\label{eq::normequiC}
        \| \eps (u_h) \|^2_T + \| \curl u_h - \omega_h\|_T^2 \sim&~ \| \dev \nabla u_h - \Pi^0 \kappa(\omega_h) \|_T^2 \\
        + &~ h^2 \| \div \omega_h \|_T^2 + \| \div u_h \|_T^2.  \end{split} 
    \end{align}
  \end{subequations}
\end{lemma}
\begin{proof}
  The first equivalence follows by standard scaling arguments (by
  equivalence of norms in the lowest order Raviart-Thomas space).
  Equivalence \eqref{eq::normequiB}
  also follows by local scaling  arguments and  \cite[eq.~(4.14)]{mcsII}.
  We continue on to prove
  \eqref{eq::normequiC}. Applying  the Pythagoras theorem twice, 
  \begin{align}
  \nonumber 
    \| \eps(u_h) \|_T^2 +
    \frac 1 2 \| \curl u_h - \omega_h\|_T^2
    &= \| \nabla u_h - \kappa(\omega_h)\|_T^2 \\ \label{eq:eps-dev-etc}
    &= \| \dev \nabla u_h - \kappa(\omega_h)\|_T^2 + \frac 1 3 \| \div u_h\|_T^2.
\end{align}
We also have, due to~\eqref{eq::normequiB},  
\begin{equation}
  \label{eq:dev-omega--curl}
\begin{aligned}
  h^2 \| \div \omega_h \|^2_T \sim&~  h^2 \| \curl (\kappa (\omega_h)) \|^2_T = h^2\big\| \curl \big(\kappa(\omega_h-\curl u_h) \big)\big\|^2_T \\
  \lesssim &~ \| \omega_h - \curl u_h \|_T^2.
\end{aligned}  
\end{equation}
Here we have used an inverse inequality and the observation that
derivatives of $\curl u_h \in P^0(T)$ vanish. Combining~\eqref{eq:eps-dev-etc}, \eqref{eq:dev-omega--curl} and the continuity of the $L^2$ projection, we conclude that the right side of \eqref{eq::normequiC} can be bounded by
the left side.

For the reverse inequality,
\begin{align}
  \nonumber 
  \| \dev \nabla u_h - \kappa(\omega_h)\|^2_T
  &= \|\Pi^0( \dev \nabla u_h -  \kappa(\omega_h))\|_T^2
    + \| (\id - \Pi^0) \kappa(\omega_h)\|_T^2 \\ \nonumber 
    &\lesssim \| \dev \nabla u_h - \Pi^0 \kappa(\omega_h)\|_T^2 + h^2 | \kappa(\omega_h)|_{H^1(T)}^2 \\ \label{eq:dev-full-Pi}
    &\sim \| \dev \nabla u_h - \Pi^0 \kappa(\omega_h)\|_T^2 + h^2 \| \div \omega_h\|_T^2.
\end{align}
Here, we used that $\dev \nabla u_h \in P^0(T)$, a standard
approximation estimate for the $L^2$ projection, followed by
\eqref{eq::normequiB}. The proof is then concluded using~\eqref{eq:eps-dev-etc}.
\qqed
\end{proof}

\begin{lemma} \label{lemma::normequi}
  For all $(u_h, \hat{u}_h, \omega_h)\in U_h$,
  \[
    \nhe{ u_h, \hat u_h, \omega_h}^2
    \,\sim\,
    \| u_h, \hat u_h, \omega_h \|^2_{\eps}
    \,\sim\,
    \| u_h, \hat u_h, \omega_h \|^2_{U_h} + \| \div u_h \|_0^2 + h^2 \|\div \omega_h \|^2_0. 
  \]
\end{lemma}
\begin{proof}
    This is a direct consequence of Lemma \ref{lemma::normequi_tech}.
    \qqed
\end{proof}

\subsection{Interpolation operators}\label{sec::interp}

In subsequent sections we will require the interpolation operators
into the spaces in \eqref{def::fespaces_V}-\eqref{def::fespaces_Q}, denoted by 
\[
  \IV: H^1_{n, D}(\mesh)\rightarrow V_h,\quad
  \IW: H^1_{n, D}(\mesh)\rightarrow W_h,\quad
  \IVh: L^2(\facets, \rr^3)\rightarrow \hat{V}_h, \quad
  \IS : \Sigma \rightarrow \Sigma_h,
\]
where $\Sigma = \{\tau \in H^{1}(\mesh, \dd):\jump{\tau}_{nt} = 0\}$.
Of course, the natural interpolation for $Q_h$, denoted by  $\IQ : L^2(\Omega) \rightarrow Q_h$, is simply the $L^2$-orthogonal projection.
The definitions and properties of the remaining interpolants
are summarized in this subsection.

An $H(\div)$-interpolation into $V_h$, denoted by $\IV: H^1_{n,
D}(\mesh)\rightarrow V_h,$ is defined using the standard degrees of
freedom (see e.g., \cite[Proposition 2.3.2]{brezzi2012mixed}): 
\begin{align}
   \quad \int_{F} (u - \IV u)_n \,q\, \ds = 0 \quad \textrm{for all }
   q\in P^1(F) \textrm{ and } F\in\facets.   \label{def::interp_V}
\end{align}
A well-known consequence of~\eqref{def::interp_V} is that
\begin{equation}
  \label{eq:commut}
  \div (\IV u) = \IQ \div u,   
\end{equation}
for all $u$ in the domain of $\IV$.
The interpolant $\IW: H^1_{n, D}(\mesh)\rightarrow W_h$, defined by
$
  ((\omega - \IW \omega)_n, q)_F = 0  
  $
for  all $ q\in P^0(F)$ and all $F\in\facets,$ is also standard.
The 
interpolation operator for the stress space
$\IS : \Sigma \rightarrow \Sigma_h$,
borrowed from~\cite{mcsI}, is defined by 
\begin{align}
  & \int_F(\IS\sigma - \sigma)_{nt} \cdot q \ds= 0, && \Forall q\in P^0(F, \rr^3) \text{ with }q_n=0,~\Forall F\in\facets,  \label{def::interp_Sigma_vol} \\
  & \int_T(\IS\sigma - \sigma) : q \dx= 0, && \Forall q\in P^0(T, \dd),~\Forall T\in\mesh. \label{def::interp_Sigma_edge}
\end{align}
Finally, the  tangential $L^2$-projection on facets, 
$\IVh: L^2(\facets, \rr^3)\rightarrow \hat{V}_h$ is defined as usual by $((\hat{u} - \IVh \hat{u})_t,  q)_F = 0$
for all $q\in P^0(F, \rr^3)$ with $ q_n=0$ on all  $F\in\facets.$

To note the salient approximation properties of these interpolants, first
observe that for a $u \in
H^1(\Omega, \rr^3) \cap H^2(\mesh)$, we have $\curl(u) \in
H^1_{n, D}(\mesh)$. Hence $(\IV u, \IVh u_t, \IW \curl (u))$ is in $U_h$ and 
using standard scaling arguments and the Bramble-Hilbert lemma, we get
\begin{align} \label{eq::approx_U}
  \nhe{u - \IV u, \; u_t -  \IVh  u_t,  \;
  \curl u - \IW \curl u }^2 & + \|u - \IV u, \; u_t -  \IVh  u_t\|_\nabla^2 
                               \nonumber
  \\
  &
    +
    h \|  \veps( \IV u - u)_{nt} \|_{\partial \mesh}^2
  \lesssim\, h^2 \| u \|^2_{H^2(\mesh)}.
\end{align}
Also recall that  \cite[Theorem~5.8]{mcsI} implies that for all $\sigma \in \Sigma,$ \begin{align}\label{eq::approx_Sigma}
   \| \sigma - \IS \sigma \|_0^2 +   h \left\| (\sigma - \IS \sigma)_{nt} \right\|_{\partial \mesh}^2 & \lesssim\; h^2 \| \sigma \|^2_{H^1(\mesh)}.
\end{align}

\section{An $H(\div)$-conforming velocity--vorticity HDG
scheme}\label{sec::hdg}

\subsection{Derivation of the HDG method}

To derive our new HDG scheme for \eqref{eq::stokesweak}, let $u,p$ be
a sufficiently smooth exact solution of \eqref{eq::stokes}.  (A
sufficient smoothness condition  is quantified in
Lemma~\ref{lem::hdg_consistency} below.)  Let $v_h \in V_h$.
Then, multiplying~\eqref{eq::stokes-a} by $v_h$ and integrating by
parts on each element,
\begin{subequations}
  \begin{equation}
    \label{eq:hdg-derv-1}
    \begin{aligned}
      (f, v_h) = (-\div(\nu\eps(u)) + \nabla p, v_h)
      & = -(p, \div v_h) + \sum\limits_{T \in \mesh}  \int_T \nu\eps(u): \eps(v_h) \dx
      \\      
      & + \sum\limits_{T \in \mesh}
      \int_{\partial T \setminus \Gamma_N} (p - \nu \eps(u))n   \cdot v_h \ds,
    \end{aligned}
  \end{equation}
  where we used the symmetry of $\eps(u)$ and
  the boundary condition~\eqref{eq::stokes-d}.
  Since $p$ is smooth,  $\jump{v_h}_n=0$ on $\facets_{0, D}$,
  \begin{equation}
    \label{eq:hdg-derv-1.1}  
    0 = 
    -\sum\limits_{T \in \mesh}
    \int_{\partial T \setminus \Gamma_N} p n \cdot v_h \ds.
  \end{equation}
  Let $\hat v_h \in \hat V_h.$
  Since $\hat v_h$ is single-valued on all facets,
  $\hat{v}_h=0$ on $\Gamma_D$ (see \eqref{def::fespaces_Vhat}),
  and $\eps(u)$ is continuous across interior facets,
  \begin{equation}
    \label{eq:hdg-derv-2}  
    0 =
    \sum\limits_{T \in \mesh}
    \int_{\partial T \setminus \Gamma_N } \nu \veps(u) n \cdot \hat v_h \ds.
  \end{equation}
\end{subequations}
  Adding~\eqref{eq:hdg-derv-1}--\eqref{eq:hdg-derv-2},
\[    \begin{aligned}
    (f, v_h) 
    & =  -(p, \div v_h) + \sum\limits_{T \in \mesh}  \int_T \nu\eps(u): \eps(v_h) \dx
    + \int_{\partial T \setminus \Gamma_N} \nu \eps(u)n   \cdot ( \hat v_h - v_h) \ds,
  \end{aligned}
\]
Since $(\hat v_h)_t = \hat v_h$, $\jump{v_h}_n=0$ on $\facets_{0, D}$
and $\veps(u)$ is smooth, we may replace $( \hat v_h - v_h)$ by its
tangential component $( \hat v_h - v_h)_t$ in the last term above.
Furthermore, on $\Gamma_N,$ we have
$\veps(u) n \cdot (\hat v_h - v_h)_t = \veps(u)_{nt} \cdot (\hat v_h -
v_h)_t = 0$ since the tangential part of~\eqref{eq::stokes-d} shows
that $\eps(u)_{nt}=0$ on $\Gamma_N.$ Hence we may also replace
$\partial T \setminus \Gamma_N$ by $\partial T$ in the last
term. Thus,
  \begin{subequations}
    \begin{equation}
      \label{eq:hdg-derv-3}
      \begin{aligned}
        (f, v_h) 
        & =  -(p, \div v_h)
        + \sum\limits_{T \in \mesh}  \int_T \nu\eps(u): \eps(v_h) \dx        + \int_{\partial T} \nu \eps(u)n   \cdot ( \hat v_h - v_h)_t \ds.
      \end{aligned}
    \end{equation}
Next, let $\omega = \curl(u)$
and $\hat u = u_t$ on each element boundary $\partial T$.
Then, obviously, 
\begin{align} \label{eq::derivhdg}
  0 &= \sum\limits_{T \in \mesh}\int_{\partial T} \nu\eps(v_h)n  \cdot (\hat u - u)_t \ds +  \sum\limits_{T \in \mesh}\frac{\nu\alpha}{h} \int_{\partial T} \Pi^0(\hat u - u)_t \cdot \Pi^0(\hat v_h - v_h)_t \ds,
  \\\label{eq::const_term_curl}
  0 & =  \sum\limits_{T \in \mesh}h \int_{\partial T} \nu (\curl u  - \omega)_n (\curl v_h - \eta_h)_n \dx,
\end{align}
\end{subequations}
for any test function $\eta_h \in W_h$ and constant $\alpha > 0$,
i.e., if $u, \hat u$ and $\omega$ are replaced by $u_h, \hat u_h$ and $\omega_h$, respectively, then 
the terms on the right are consistent terms. 

Adding the
equations~\eqref{eq:hdg-derv-3}--\eqref{eq::const_term_curl}, we
obtain
\begin{subequations}
  \label{eq:exhdg}
  \begin{equation}
    \label{eq:exhdg-1}
    \nu a^{\hdg}(u, \hat u, \omega; v_h, \hat v_h, \eta_h)
    - (\div v_h, p) = (f, v_h),
  \end{equation}
  where 
  \begin{align*}
    a^{\hdg}& (z, \hat z, \theta;\, v_h, \hat v_h, \eta_h)\,
              := \left(\eps(z),  \eps(v_h)\right)_{\mesh}
    \\
            &      
      +
              \left(\eps(z) n,  (\hat v_h - v_h)_t\right)_{\partial \mesh}
      +
              \left( (\hat z - z)_t, \eps(v_h) n\right)_{\partial \mesh}
    \\
            &
              + \frac{\alpha}{h}
              \left(\Pi^0(\hat z - z)_t,
              \Pi^0(\hat v_h - v_h)_t\right)_{\partial \mesh}
              + h   \left((\curl z - \theta)_n,
              (\curl v_h - \eta_h)_n\right)_{\partial \mesh}.
  \end{align*}
  Here and throughout,
  $(\cdot, \cdot)_\mesh = \sum_{T \in \mesh} (\cdot, \cdot)_T$ and
  $(\cdot, \cdot)_{\partial\mesh} = \sum_{T \in \mesh} (\cdot,
  \cdot)_{\d T}$, extending our prior analogous norm notation to inner products.
  Of course, from \eqref{eq::stokes-b}, we also have
\begin{equation}
   \label{eq:exhdg-2}
  (\div u, q_h) = 0,  
\end{equation}
\end{subequations}
for all $q_h \in Q_h.$ Equations
\eqref{eq:exhdg-1}--\eqref{eq:exhdg-2}, after replacing
$(u, \hat u, \omega)$ by $(u_h, \hat u_h, \omega_h)$, yield the
following discrete formulation: find
$(u_h, \hat u_h, \omega_h) \in U_h $ and $p_h \in Q_h$ such that
\begin{subequations} \label{eq::hdgstokes}
  \begin{alignat}{2} 
    \nu \,a^{\hdg}(u_h,\hat u_h, \omega_h;v_h, \hat v_h, \eta_h) - (p_h, \div v_h) &= (f,v_h), \label{eq::hdgstokesA}\\
    -(\div u_h, q_h) &= 0,  \label{eq::hdgstokesB}
  \end{alignat}
\end{subequations}
for all $ (v_h, \hat v_h, \eta_h) \in U_h$ and $ q_h \in Q_h$. Note
that this method enforces $\Pi^0(u_h)_t=0$ on $\Gamma_D$ as a
consequence of how the last term of~\eqref{eq::derivhdg} manifest in
the method.  Due to the Dirichlet conditions built into $W_h$ (see
\eqref{def::fespaces_W}) the method also penalizes
$\|\Pi^0\curl(u_h)_n\|_{\Gamma_D}$ through the manifestation of the
consistent term \eqref{eq::const_term_curl} in the method.
System~\eqref{eq::hdgstokes} may be thought of as a nonconforming HDG
discretization of the standard weak form~\eqref{eq::stokesweak}.

Note that $a^{\hdg}(u, \hat u, \omega; v_h, \hat v_h, \eta_h)$ is well
defined for any $(v_h, \hat v_h, \eta_h) \in U_h$ and any
$(u, \hat u, \omega) \in \Ureg,$ where
\begin{subequations}\label{def::ureg}
  \begin{align}
    \Ureg &:= (H_{0,D}^1(\Omega) \cap H^2(\mesh)) \times L^2(\facets) \times H^1_{n, D}(\mesh), \\
    \Qreg &:= Q \cap H^1(\mesh). 
  \end{align}
\end{subequations}

\begin{lemma}[Consistency of the HDG method] \label{lem::hdg_consistency}
  Suppose the exact solution
  $(u,p)$ of \eqref{eq::stokesweak} is regular enough so that 
  $u$, together with $\hat u = u_t$ on facets and $\omega = \curl u$, satisfies 
  $ (u, \hat u, \omega)  \in \Ureg$ and suppose $p \in \Qreg$. Then any 
  $((u_h, \hat u_h, \omega_h),p_h) \in U_h \times Q_h$ solving
  \eqref{eq::hdgstokes} satisfies
  \begin{align*}
    \nu a^{\hdg}(u-u_h, u_t - \hat u_h, \omega - \omega_h; v_h, \hat{v}_h, \eta_h) - (\div v_h,p - p_h) = 0,
  \end{align*}
  for all
  $(v_h, \hat v_h, \eta_h) \in U_h$.
\end{lemma}
\begin{proof}
  This follows by subtracting~\eqref{eq::hdgstokes} from~\eqref{eq:exhdg}.
  \qqed
\end{proof}

\subsection{Pressure robust error analysis of the HDG scheme}

We follow the usual mixed method approach and proceed to combine
continuity and coercivity of $a^{\hdg}$ with a discrete Stokes inf-sup
condition, or the LBB~\cite{brezzi2012mixed} estimate.  The latter
implies the stability of~\eqref{eq::hdgstokes}, which also implies its
unique solvability. We begin by noting that by local
scaling arguments, there is a mesh-independent $c_1$ such that 
\begin{equation}
  \label{eq:epsbdr-int-trivial}
  h \| \eps(v_h) \|^2_{\partial T} \le c_1 \| \eps(v_h)\|^2_T,
  \qquad v_h \in V_h,  \quad T \in \mesh,
\end{equation}
since $\veps(v_h)$ is constant on $T$. For the same reason, $\Pi^0$
may be introduced into the second and third terms in the definition of
$a^{\hdg}(u_h,\hat u_h, \omega_h;v_h, \hat v_h, \omega_h)$, e.g.,
\begin{equation}
  \label{eq:hdg-trivial-Piintro}
  \left( \veps(u_h)n, (\hat v_h - v)_t \right)_{\partial \mesh} =
  \left( \veps(u_h)n, \Pi^0(\hat v_h - v)_t \right)_{\partial \mesh}.
\end{equation}
Let $\nhep{ u, \hat u, \omega}^2 = \nhe{ u, \hat u, \omega}^2
+ h \|\eps(u)_{nt}\|_{\partial\mesh}^2 +
h^{-1}\| (\id - \Pi^0) (u - \hat u) \|_{\partial \mesh}^2.$ 
\begin{lemma}[Continuity of of $a^{\hdg}$]
  \label{lemma::hdg_continuity}
  For any $(u, \hat u, \omega) \in \Ureg$,
  $(u_h,\hat u_h, \omega_h) \in U_h, $
  $(v_h, \hat v_h, \omega_h) \in U_h$ and $q_h \in Q_h$,
  \begin{align}
    \label{eq::hgd_continuity-0}
    \nu a^{\hdg}(u,\hat u, \omega;v_h, \hat v_h, \omega_h)
    &\lesssim \nu \nhep{u, \hat u, \omega} \nhep{v_h, \hat v_h, \eta_h}, \\
    \label{eq::hgd_continuity}
    \nu a^{\hdg}(u_h,\hat u_h, \omega_h;v_h, \hat v_h, \omega_h)  &\lesssim \nu \nhe{u_h, \hat u_h, \omega_h} \nhe{v_h, \hat v_h, \eta_h}, \\
    \label{eq::hgd_continuity-div}
    (\div u_h, q_h) &\lesssim \nhe{u_h, \hat u_h, \omega_h}\|q_h\|_0.
  \end{align}
\end{lemma}
\begin{proof}
  Inequality~\eqref{eq::hgd_continuity-0} follows from Cauchy-Schwarz
  inequality, while~\eqref{eq::hgd_continuity} follows by additionally employing
  \eqref{eq:epsbdr-int-trivial}
  and~\eqref{eq:hdg-trivial-Piintro}. The
  estimate~\eqref{eq::hgd_continuity-div} is a consequence of
  $ \frac 1 3 \|\div u_h \|^2_T = \| \veps(u_h) \|_T^2 - \| \dev
  \veps(u_h) \|^2_T \le  \| \veps(u_h)\|_T^2.$
  \qqed
\end{proof}

\begin{lemma}[Coercivity of $a^{\hdg}$]\label{lemma::hdg_coercivity}
  There is a mesh-independent $\alpha_0>0$ such that for all 
  $\alpha >\alpha_0$ and all
  $(u_h, \hat{u}_h,
  \omega_h)\in U_h$, 
  \begin{align}\label{eq::hgd_coercivity}
    \nu a^{\hdg}(u_h,\hat u_h, \omega_h;u_h, \hat u_h, \omega_h)  \gtrsim
    \nu \nhe{u_h, \hat u_h, \omega_h}^2.
  \end{align}
\end{lemma}
\begin{proof}
  By~\eqref{eq:hdg-trivial-Piintro} and Young's inequality with any $\beta>0$, 
  \begin{align*}
    a^{\hdg}(u_h,\hat u_h, \omega_h;u_h, \hat u_h, \omega_h)
    & \ge
      \| \eps(u_h) \|_\mesh^2
      -
      \Big(
      \beta h\| \veps (u_h) n \|_{\partial \mesh}^2 + 
      \frac{1}{\beta h} \| \Pi^0 (u_h - \hat u_h)_t\|^2_{\partial \T}
      \Big) \\
    & +
      \alpha h^{-1}
      \| \Pi^0 (u_h - \hat u_h)_t\|^2_{\partial \T}
      + h \| (\curl u_h - \omega_h)_n\|_{\partial \T}^2.
  \end{align*}
  Hence using~\eqref{eq:epsbdr-int-trivial}, and choosing, say $\beta =
  1/(2c_1)$ and $\alpha = 2/\beta$, \eqref{eq::hgd_coercivity} follows.
  \qqed
\end{proof}

\begin{lemma}[LBB condition for the HDG method] \label{lem::hdgstokeslbb} For any
  $p_h\in Q_h$ there exists a $(v_h, \hat v_h, \eta_h)\in U_h$ with
  $\nhe{v_h, \hat v_h, \eta_h}\lesssim\|p_h\|_0$ and
  $\div v_h = p_h$. Consequently,
  \begin{align}\label{eq::hdg_stokes_lbb}
    \sup_{(v_h, \hat v_h, \eta_h) \in U_h} \frac{(\div v_h, p_h)}
    { \nhe{v_h, \hat v_h, \eta_h}} \gtrsim \|p_h\|_0.
  \end{align}
\end{lemma}
\begin{proof}
  By classical results \cite{girault2012finite}, there exists a
  $u \in H^1(\Omega)$ such that
  \begin{equation}
    \label{eq:lbbhdg-1}
    \div v = p_h, \qquad 
    \| v \|_{H^1(\Omega)} \lesssim \| p_h \|_0.     
  \end{equation}
  Put $v_h = \IV v$ and $\hat v_h = \IVh v$ on each facet. Then,
  \eqref{eq:lbbhdg-1} and~\eqref{eq:commut} imply
  $ \div v_h = \div( \IV v) = \IQ \div v = p_h.$ Moreover, (as
  alluded to in \cite{LS_CMAME_2016}) it is easy to show that
  \begin{equation}
    \label{eq:lbbhdg-2}
    \left\| v_h, \hat v_h \right\|_\nabla \lesssim \| v \|_{H^1(\Omega)}.
  \end{equation}
  Choose $\eta_h\in W_h$ as in \eqref{eq::korn_hdg_reverse} of
  Lemma~\ref{lem::korn_hdg}. Then, by~\eqref{eq:lbbhdg-1}--\eqref{eq:lbbhdg-2}, 
  \[
    \nhe{v_h ,\hat v_h, \eta_h} \lesssim \|v_h, \hat v_h\|_{\nabla} \lesssim
    \| v \|_{H^1(\Omega)} \lesssim \|p_h\|_0,
  \]
  concluding the proof.
  \qqed
\end{proof}

\begin{theorem}[Error estimates for the HDG method]
  \label{thm::hdg_apr_error}
  Let $u, \hat u, \omega, p$ denote the exact solution that satisfies
  the regularity assumption of Lemma~\ref{lem::hdg_consistency} and
  let $((u_h, \hat u_h, \omega_h),p_h) \in U_h \times Q_h$ be the
  discrete solution of \eqref{eq::hdgstokes}. Then the errors in
  $u_h, \hat u_h, \omega_h$ can be bounded independently of the
  pressure error by 
  \begin{align}\label{eq::hdg_apr_error_vel}
    \nhe{u-u_h, u_t-\hat{u}_h, \omega -\omega_h} \lesssim h \|u\|_{H^2(\mathcal{T})}.
 \end{align}
 Furthermore, the pressure error satisfies
  \begin{align}\label{eq::hdg_apr_error_pres}
    \nu^{-1} \| p - p_h\|_0 \lesssim h ( \|u\|_{H^2(\mathcal{T})}
    + \nu^{-1} \| p \|_{H^1(\mesh)}).
  \end{align}
\end{theorem}
\begin{proof}
  Let $E= (u - u_h, \hat u - \hat u_h, \omega - \omega_h)$ and
  $E_h = (\IV u - u_h, \IVh \hat u - \hat u_h, \IW\omega - \omega_h)$.
  Then  $\EE = E - E_h$ represents the interpolation errors.  Since $E_h \in U_h$,
  \begin{align*}
    \nu \nhe{E_h}^2
    & \lesssim \nu  a^{\hdg}(E_h; E_h)  = \nu a^{\hdg}(E - \EE ; E_h) 
    && \text{ by Lemma~\ref{lemma::hdg_coercivity}}
    \\    
    & = (\div (\IV u - u_h), p  -p_h)  - \nu a^{\hdg}(\EE ; E_h)
    && \text{ by Lemma~\ref{lem::hdg_consistency}.}
  \end{align*}
  By~\eqref{eq:commut},  $\div (\IV u) = \IQ \div u = 0$. Moreover,   by  \eqref{eq::hdgstokesB}, $\div u_h = 0$. Hence
  \begin{equation}
    \label{eq:cnty-use}
    \nu \nhe{E_h}^2
    =  - \nu a^{\hdg}(\EE ; E_h)
    \lesssim \nu \nhep{\EE} \nhep{E_h},
  \end{equation}
  by Lemma~\ref{lemma::hdg_continuity}. Now we claim that
  \begin{equation}
    \label{eq:stronger-nrm-2weaker}
    \nhep{E_h} \lesssim \nhe{E_h}.    
  \end{equation}
  To see this, first note that local scaling arguments give
  \begin{equation}
    \label{eq:1200}
    h^{-1} \| (I - \Pi^0) v_h \|^2_{\partial\mesh}
    \lesssim \| \nabla v_h \|^2_\mesh,
  \end{equation}
  for any $v_h \in V_h$.  Then, letting
  $E_h^u = \IV u - u_h$, $E_h^{\hat u} = \IVh u - \hat u_h$, note that
  on each facet,
  $(I - \Pi^0) \left( E_h^u - E_h^{\hat u} \right) = (I - \Pi^0)E_h^u.$
  Hence the extra terms in $\nhep{E_h}^2$ that are not in
  $\nhe{E_h}^2$ can be bounded by applying~\eqref{eq:1200} and
  \eqref{eq:epsbdr-int-trivial} with $v_h = E_h^u$ to get
  \[
    \nhep{E_h}^2 \lesssim \nhe{E_h}^2 + \| E_h^u, E_h^{\hat u}\|_\nabla^2
    \lesssim \nhe{E_h}^2,
  \]
  by Lemma~\ref{lem::korn_hdg}.  This
  proves~\eqref{eq:stronger-nrm-2weaker}. Using~\eqref{eq:stronger-nrm-2weaker}
  in~\eqref{eq:cnty-use}, we conclude that
  $\nhe{E_h} \lesssim \nhep{\EE}$. Combining with triangle inequality,
  \begin{equation}
    \label{eq:Ebound}
    \nhe{E} \le \nhe{\EE} + \nhe{E_h} \lesssim \nhep{\EE}
    \lesssim h \| u \|_{H^2(\mesh)},
  \end{equation}
  where we have applied~\eqref{eq::approx_U} in the last step.  This
  proves~\eqref{eq::hdg_apr_error_vel}.

  For the pressure estimate, we begin with triangle inequality
  and Lemma~\ref{lem::hdgstokeslbb}:
\begin{align*}
  \nu^{-1} \| p - p_h \|_0 &\le \nu^{-1} \| p - \IQ p \|_0
                             +\nu^{-1}  \| \IQ p - p_h \|_0 \\
  &\lesssim \nu^{-1} \| p - \IQ p \|_0 +   \sup_{(v_h, \hat{v}_h, \eta_h) \in U_h} \frac{\nu^{-1}(\div v_h, \IQ p -  p_h )}{ \nhe{v_h, \hat v_h, \eta_h}}.
\end{align*}
To bound the numerator of the supremum, we use  Lemma~\ref{lem::hdg_consistency}: 
\begin{align*}
  \nu^{-1}(\div v_h, \IQ p - p_h)
  &= \nu^{-1} (\div v_h, p-p_h)
    =  a^{\hdg}(E; v_h, \hat{v}_h, \eta_h)
  \\ 
  & \lesssim \nhe{E} \nhe{v_h, \hat{v}_h, \eta_h}.
\end{align*}
Hence the already proved estimate~\eqref{eq:Ebound}, together with the
standard $L^2$ projection error estimates finish the proof
of~\eqref{eq::hdg_apr_error_pres}.
\qqed
\end{proof}

\section{An MCS formulation with 
$H(\div)$-conforming vorticity} \label{sec::mcs}

In this section we derive a new mixed method for the approximation of
\eqref{eq::mixedstressstokes}, motivated by the weak
formulation~\eqref{eq::mixedstressstokesweak}. Let
$\sigma_h \in \Sigma_h$ and $(v_h,\hat v_h, \eta_h) \in U_h$. Defining
\begin{equation}
  \label{eq:div-hybrid-pairing}
\begin{aligned}
  \langle \div \sigma_h ; \, v_h,\hat v_h, \eta_h \rangle_{U_h} 
  := & \;
    (\div\sigma_h,   v_h)_\mesh
    -
    \left((\sigma_h)_{nn},  (v_h)_n \right)_{\partial \mesh}
  \\
  &
    -\left( (\sigma_h)_{nt},  (\hat v_h)_t\right)_{\partial\mesh}
    +
    \left( \sigma_h, \kappa(\eta_h)\right)_\mesh,
\end{aligned}  
\end{equation}
consider the terms on the right.  When $(\sigma_h)_{nt}$ is continuous
across element interfaces, the first two terms together realizes the
duality pairing introduced in Section~\ref{sec::prelim}, namely
$\langle \div \sigma, v_h\rangle_{\div}$, per
\cite[Theorem~3.1]{mcsII}.  The third term is used to impose the
$nt$-continuity of the viscous stress (and prior works \cite{mcsI,
mcsII, lederer2019mass} provided enough rationale to employ
$nt$-continuous finite elements for viscous stresses). Note, that
a similar nt-continuous approximation of the gradient (but not the
physical viscous stresses $\eps(u)$) was also already
considered in~\cite{MR3941890}.  Due to the Dirichlet conditions
built into $\hat{V}_h$ on $\Gamma_D$ (see \eqref{def::fespaces_Vhat}),
this term is comprised only of integrals over facets in the interior
and on $\Gamma_N$, with the latter enforcing $\sigma_{nt}=0$ in
$\Gamma_N$ as demanded by \eqref{eq::mixedstressstokes-f}.  Finally,
the last term above is used to weakly incorporate the symmetry
constraint \eqref{eq::mixedstressstokes-c}.  This technique of
imposing symmetry weakly is widely used in finite elements for linear
elasticity~\cite{MR840802, MR2336264, MR2449101,
MR2629995,MR1464150,GopalGuzma12,Stenberg1988}.

Viewing~\eqref{eq::mixedstressstokesweak} in terms of 
$\langle \div \cdot, \cdot \rangle_{U_h}$, we are
led to the following mixed method: find
$(u_h, \hat u_h, \omega_h) \in U_h$ and
$ (\sigma_h,p_h) \in (\Sigma_h \times Q_h)$ satisfying 
 \begin{subequations} \label{eq::mcsstokes}
    \begin{alignat}{2} 
      {\nu}^{-1}(\sigma_h, \tau_h) + \langle \div\tau_h; \,
      u_h,\hat u_h, \omega_h \rangle_{U_h} &= 0, \\
      -\langle \div\sigma_h; \,v_h,\hat v_h, \eta_h \rangle_{U_h}
      -(\div v_h, p_h) +  c(\omega_h,\eta_h)&= (f,v_h),
      \\
     -(\div u_h, q_h) &= 0,
   \end{alignat}
 \end{subequations}
 for all $\tau_h \in \Sigma_h,$ $(v_h, \hat v_h, \eta_h) \in U_h,$ and
 $q_h \in Q_h$,
 with the stabilizing bilinear form $c(\omega_h,\eta_h):= \nu h^2
 (\div \omega_h, \div \eta_h)_\Omega$. Note that since $\omega_h$
 approximates the vorticity $\omega = \curl(u)$, we have
 $\div \omega=0$, so $c(\cdot, \cdot)$ is a consistent addition.
 Although the formulation \eqref{eq::mcsstokes} is very similar to the
 formulations from \cite{mcsI} and \cite{mcsII}, note the following
 differences. First, while the $nt$-continuity of viscous stresses was
 built into the spaces in~\cite{mcsI,mcsII}, now it is incorporated as
 an equation of the method by the well-known hybridization
 technique.  Second, although we use the same local
 stress finite element space as in \cite{mcsI}, we use the
 weak symmetric setting from \cite{mcsII}. In the latter, the Lagrange
 multiplier for the weak symmetry constraint was given by an
 element-wise discontinuous approximation, whereas here it is in 
 the div-conforming~$W_h$.

\subsection{Stability of the MCS method}

From the terms in~\eqref{eq::mcsstokes}, we anticipate that the norms
$\| \cdot \|_{U_h}$ and $\| \cdot \|_\eps$ are more natural for the
analysis of the MCS method (in contrast to the HDG method).
The latter appears in the next lemma.

\begin{lemma}[Continuity of MCS formulation]
  \label{th:mcscontinuity}
  The bilinear forms in \eqref{eq::mcsstokes} are
  continuous in the sense that for all $\sigma_h ,\tau_h \in \Sigma_h$,
  $p_h \in Q_h$, $\eta_h \in W_h$ and $ (u_h,\hat u_h, \omega_h) \in
  U_h,$ in addition to the obvious estimates
  \[
    \nu^{-1}(\sigma_h,\tau_h)
    \lesssim {{\nu}}^{-1/2}  \| \sigma_h \|_0\, {{\nu}}^{-1/2}  \| \tau_h \|_0,
    \quad \textrm{and} \quad
    c(\omega_h, \eta_h)
    \lesssim \nu h^2 \| \div\omega_h \|_0\, \|\div\eta_h\|_0,
  \]
  the following estimates  hold:
  \begin{subequations}
    \label{eq:msc-cont}
    \begin{align}
      \label{eq:msc-cont-c}
      (\div u_h, p_h)
      &\lesssim \| u_h,\hat u_h, \omega_h\|_{\eps} \,\| p_h \|_{0},
      \\     \label{eq:msc-cont-d}
      \langle \div\sigma_h; \, u_h,\hat u_h, \omega_h \rangle_{U_h}
      &\lesssim \| \sigma_h \|_0 \,\|u_h,\hat u_h, \omega_h \|_{\eps}.
    \end{align}
  \end{subequations}
\end{lemma}
\begin{proof}
  Inequality~\eqref{eq:msc-cont-c} is proved just
  like~\eqref{eq::hgd_continuity-div}. To prove~\eqref{eq:msc-cont-d},
  let us first note an equivalent and more compact form of $\langle
  \div\sigma_h;\, v_h,\hat v_h, \eta_h \rangle_{U_h}$ obtained by
  integrating \eqref{eq:div-hybrid-pairing} by parts (see e.g,
  \cite[eq.~(3.11)]{mcsII}), namely
  \begin{align} \label{eq::mcsbblf}
    \langle \div\sigma_h;\, 
    & v_h,\hat v_h, \eta_h \rangle_{U_h} 
      = -(\sigma_h, \nabla v_h - \kappa(\eta_h))_\mesh
      + ( (\sigma_h)_{nt},  (v_h - \hat v_h)_t)_{\partial \mesh}.
  \end{align}
  Using~\eqref{eq::mcsbblf}, the fact that $\sigma_h$ is trace-free,
  the Cauchy-Schwarz inequality, and the following estimate (which follows by a local scaling argument using a specific mapping mentioned in the beginning of \S\ref{sec::fem_ne_io}), 
  \begin{equation}
    \label{eq:easy-to-prove}
      h^{1/2} \| (\sigma_h)_{nt} \|_{\partial \mesh} \lesssim \| \sigma_h
  \|_\mesh,
  \end{equation}
  we get 
  \begin{align*}
    \langle \div\sigma_h; \, u_h,\hat u_h, \omega_h \rangle_{U_h}
    & \lesssim \| \sigma_h \|_0 \left(
      \| \dev \nabla u_h - \kappa(\omega_h) \|^2_\mesh
      + h^{-1} \| \Pi^0( u_h - \hat u_h\|_{\partial \mesh}^2
      \right)^{1/2}
    \\
    & \lesssim \| \sigma_h\|_0
      \left( \| u_h,\hat u_h, \omega_h\|_{U_h}^2
      + h^2 \| \div\omega_h \|_\mesh^2\right)^{1/2},
  \end{align*}
  where the last inequality is due to the same argument as in
  \eqref{eq:dev-full-Pi}. Thus~\eqref{eq:msc-cont-d} follows from
  Lemma~\ref{lemma::normequi}.
  \qqed
\end{proof}

\begin{lemma}\label{lemma:lbb_mcs} For any $(u_h, \hat{u}_h, \omega_h)
  \in U_h$ there exists a $(\tau_h,q_h) \in \Sigma_h \times Q_h$ satisfying
  \begin{align}
    \label{eq:lbb-mcs-1}
    \|\tau_h\|_0 + \| q_h \|_0
    & \lesssim \|u_h, \hat{u}_h, \omega_h\|_{U_h}
      + \| \div u_h \|_0,
    \\ \label{eq:lbb-mcs-2}
    \langle \div\tau_h; \, u_h,\hat u_h, \omega_h \rangle_{U_h} - (\div u_h,q_h)  & \gtrsim (\|u_h, \hat{u}_h, \omega_h\|_{U_h} + \| \div u_h \|_0)^2. 
  \end{align}   
\end{lemma}
\begin{proof}
  For each element $T \in \mesh$ and each facet $F\subset \partial T$,
  there are matrix fields $S^F_0, S^F_1$, supported on $T$, with the
  following properties: on $T$, both $S^F_0, S^F_1$ are constant
  matrices in $\dd$, their boundary trace  $(S^F_i)_{nt}|_F$,
  for $i \in \{0, 1\},$
  are constant unit-length vector fields on $F$ that form a basis for  the tangent
  space $n_F^\perp$, and $(S^F_i)_{nt}|_{F'}$ vanishes on all other
  facets $F' \ne F$ in $\mathcal{F}_h$. Such matrix fields are
  exhibited in~\cite[Lemma~5.1]{mcsI}.
  Given  any $(u_h, \hat{u}_h, \omega_h) \in U_h$, set
  \begin{align*}
    \tau_h^0 &:= \sum_{T\in\mesh}\sum_{F\subset\partial T} \sum_{i \in \{0,1\}}- ( S^F_i : \Pi^0\dev(\nabla u_h - \kappa(\omega_h))) \; \lambda^F S^F_i,
    \\
    \tau_h^1 &:= \sum_{T\in\mesh}\sum_{F\subset\partial T} \sum_{i \in \{0,1\}} \frac{1}{\sqrt{h}}\Pi^0(\hat u_h - u_h)_t\; S^F_i,
  \end{align*}
  where $\lambda_F$ is the linear barycentric coordinate function
  associated to the vertex opposite to the facet $F$. Since
  $\lambda^FS^F_i$ has a vanishing $nt$-trace and
  $\Pi^0\dev(\nabla u_h - \kappa(\omega_h)) \in \dd$, we see that
  $\tau_h = \gamma_0 \tau_h^0 + \gamma_1 \tau_h^1$, for any
  $\gamma_0,\gamma_1 \in \rr$, is an element of $\Sigma_h$. Also set
  $q_h = -\div u_h,$ so that
  $-(\div u_h, q_h) = \| \div u_h \|_0^2$. For these choices,
  \eqref{eq:lbb-mcs-1} obviously holds as long as $\gamma_i$ is chosen
  independent of $h$ and $\nu$. Indeed, such $\gamma_i$ can be chosen
  to also ensure that
  \[
    \langle \div\tau_h;\, u_h,\hat u_h, \omega_h \rangle_{U_h}
    \gtrsim \|u_h,\hat{u}_h, \omega_h\|_{U_h}^2,
  \]
  so that~\eqref{eq:lbb-mcs-2} also holds. This follows from an  argument which (we
  omit and) is similar to that detailed in \cite[Lemma~6.5]{mcsII}, proceeding 
  simply by appropriately combining Young and Cauchy-Schwarz
  inequalities.
  \qqed
\end{proof}

The combined bilinear form of the MCS method~\eqref{eq::mcsstokes} is given by 
\begin{align*} 
  B(\sigma_h, u_h, \hat{u}_h, \omega_h, p_h; \tau_h, v_h, \hat{v}_h, \eta_h, q_h) := 
  &  \nu^{-1}(\sigma_h, \tau_h)
    + \langle \div\tau_h; \, 
    u_h,\hat u_h, \omega_h \rangle_{U_h}\\
  &- \langle \div\sigma_h;\, v_h,\hat v_h, \eta_h \rangle_{U_h} \\
  & - (\div u_h, q_h) -(p_h, \div v_h) + c(\omega_h, \eta_h).
\end{align*}
Define a  norm on the product space $S_h = \Sigma_h \times
V_h \times \hat V_h \times W_h \times Q_h$ by
\begin{align*}
  \|\sigma_h,  u_h, \hat{u}_h, \omega_h, p_h\|_{S_h}
  := \nu^{-1/2} ( \|\sigma_h\|_0 + \|p_h\|_0 ) + \nu^{1/2} \|u_h, \hat{u}_h, \omega_h\|_{\eps}.
\end{align*}

\begin{lemma}[Inf-sup condition for MCS method]
  \label{lemma::infsup_mcs}
  For any $r = (\sigma_h, u_h, \hat{u}_h, \omega_h, p_h) \in S_h$, be
  arbitrary, there exists an
  $s\in S_h$ such that 
  \begin{align}\label{eq::infsup_mcs}
    B(r; s) &  \gtrsim  \| r\|^2_{S_h}, \quad \text{ and}
    \\ \label{eq::infsup_mcs-2}
    \| s \|_{S_h} & \lesssim \| r \|_{S_h}.
  \end{align}
\end{lemma}
\begin{proof}
  We will find the required $s$ as a sum of three terms, each in $S_h$,
  and each depending on the given $r$.  The first term is set using  
  $s^* = (\sigma_h, u_h, \hat{u}_h, \omega_h, -p_h)$, for which we  obviously have
  \begin{subequations}
    \label{eq:infsup_mcs-s*}
    \begin{align}
      \label{eq:infsup_mcs-s*-1}
      B(r, s^*) & = \nu^{-1} \|\sigma_h \|_0^2 + \nu  h^2 \| \div \omega_h\|_0^2,\\
      \label{eq::contone}
      \| s^* \|_{S_h} & \lesssim \| r \|_{S_h}.
    \end{align}
  \end{subequations}
  The second term is
  $\tilde s = (\nu \tau_h, 0, 0, 0, \nu q_h) \in S_h$, where
  $\tau_h \in \Sigma_h$ and $q_h\in Q_h$ are as in
  Lemma~\ref{lemma:lbb_mcs} obtained using the given components
  $u_h, \hat u_h, \omega_h$ of $r$.  The lemma gives some
  $\tilde C > 0$ such that
  \begin{subequations}
    \label{eq:infsup_mcs-stilde}
    \begin{align}
      \label{eq:infsup_mcs-stilde-1}
      B(r; \tilde s)
      &  = \nu^{-1} (\sigma_h,  \nu\tau_h) + \nu\langle \div  \tau_h; \,
        u_h,\hat u_h, \omega_h \rangle_{U_h} - \nu(\div u_h, q_h)
        \nonumber
      \\
      & \gtrsim
        (\sigma_h,  \tau_h) + 
        \nu \Big( \|u_h, \hat{u}_h, \omega_h\|^2_{U_h} + \| \div u_h \|_0^2\Big),
      \\
      \| \tilde s \|_{S_h}^2 
      & = \nu^{-1}( \| \nu \tau_h\|_0^2 + \| \nu q_h\|_0^2)
        \le \tilde C \nu \Big( \|u_h,
      \hat{u}_h, \omega_h\|^2_{U_h} + \| \div u_h \|^2_0\Big).
      \label{eq::conttwo}
    \end{align}
  \end{subequations}
  The third term is $s^\Delta = (0, -\nu^{-1}v_h, -\nu^{-1}\hat v_h,
  -\nu^{-1} \eta_h, 0) \in S_h$ where $(v_h, \hat v_h, \eta_h) \in
  U_h$ is as in Lemma~\ref{lem::hdgstokeslbb} obtained using the given
  component $p_h$ of $r$. The lemma implies that $\div v_h = p_h$ and
  \begin{subequations}
    \label{eq:infsup_mcs-sdelta}
    \begin{align}
      \label{eq:infsup_mcs-sdelta-1}
      B(r; s^\Delta)
      & = \nu^{-1} \| p_h \|_0^2
        - \nu^{-1} \langle \div\sigma_h;\,
        v_h,\hat v_h, \eta_h \rangle_{U_h}
        +
        \nu^{-1} c(\omega_h, \eta_h),
      \\
    \| s^\Delta \|_{S_h}^2
    & = \nu 
       \|\nu^{-1}v_h,  \nu^{-1}\hat v_h, \nu^{-1} \eta_h \|_\eps^2
      \lesssim \nu^{-1}\| p_h\|_0^2.
      \label{eq::contthree}
  \end{align}
\end{subequations}
Note that to obtain the last inequality,
we have also used Lemma~\ref{lemma::normequi}.

Now letting $\beta > 0$, a constant yet to be
chosen, put $s = \beta s^* + \tilde s + s^\Delta$.
Then, combining~\eqref{eq:infsup_mcs-s*-1},
\eqref{eq:infsup_mcs-stilde-1}
and~\eqref{eq:infsup_mcs-sdelta-1}, 
\begin{equation}
  \label{eq:1500}
  \begin{aligned}
    B(r; s)
    & \gtrsim \frac{\beta}{\nu}\|\sigma_h\|^2_0 + \beta \nu  h^2\|\div \omega_h \|^2_0  + \nu \|u_h, \hat{u}_h, \omega_h\|^2_{U_h}\\
    & \quad  + \nu \| \div u_h \|^2_0  + \frac 1 \nu \| p_h  \|_0^2
    - (\rho_1 + \rho_2 + \rho_3),
  \end{aligned}  
\end{equation}
where
$ \rho_1 = (\sigma_h, \tau_h), \rho_2 = - \nu^{-1}\langle
\div\sigma_h;\, v_h,\hat v_h, \eta_h \rangle_{U_h}, \; \rho_3 =
\nu^{-1} c(\omega_h, \eta_h).$ By \eqref{eq::conttwo} and Young's
inequality,
\begin{align*}
  \rho_1
  & \le \frac{\tilde C}{2 \nu} \|\sigma_h\|_0^2 +
  \frac{\nu}{2}  \Big( \|u_h,
  \hat{u}_h, \omega_h\|^2_{U_h} + \| \div u_h \|^2_0\Big).
\end{align*}
To bound $\rho_2$, note that by Lemma~\ref{th:mcscontinuity},
$\rho_2 \lesssim \nu^{-1} \| \sigma_h \|_0 \| v_h,\hat v_h,
\eta_h\|_\eps $, so by~\eqref{eq::contthree}, there is a $C^\Delta>0$
such that
$\rho_2 \le \nu^{-1/2} \| \sigma_h \|_0 \left(\frac 1 2 C^\Delta
  \nu^{-1} \| p_h \|_0^2\right)^{1/2}$. Thus
\begin{align*}
  \rho_2
  & 
    \le \frac{C^\Delta}{2 \nu} \| \sigma_h\|_0^2 + \frac{1}{4\nu} \| p_h\|_0^2.
\end{align*}
To bound $\rho_3$, we recall from Lemma~\ref{lemma::normequi} that
$ h\| \div \eta_h \|_0 \lesssim \| v_h,\hat v_h, \eta_h
\|_{\eps}$. Hence by~\eqref{eq::contthree}, there is a $C'>0$ such
that
$\rho_3 \le (\nu^{1/2}h \| \div \omega_h \|_0) \left( \frac 1 2 C'
  \nu^{-1} \| p_h \|_0^2 \right)^{1/2}$, so
\begin{align*}
  \rho_3 \le \frac{C'\nu}{2} h^2 \| \div \omega_h \|_0^2
  + \frac{1}{4 \nu} \|p_h \|_0^2.
\end{align*}
Using these  estimates for $\rho_i$ in \eqref{eq:1500},
\begin{align*}
    B(r; s)
  &
    \gtrsim \frac{2\beta - (\tilde C + C^\Delta)}{2\nu}
    \|\sigma_h\|^2_0 + \frac{2\beta - C'}{2}\nu  h^2\|\div \omega_h \|^2_0
  \\
  & \quad + 
    \frac{\nu}{2} \|u_h, \hat{u}_h, \omega_h\|^2_{U_h} +
    \frac{\nu}{2} \| \div u_h \|^2_0 + \frac {1}{2\nu} \| p_h  \|_0^2.
\end{align*}
Since $\tilde C, C^\Delta$ and $C'$ are mesh-independent constants,
choosing $\beta > \max(\tilde C + C^\Delta , C')/2$ and recalling the
norm equivalence of Lemma~\ref{lemma::normequi}, we prove
\eqref{eq::infsup_mcs}. Of course, inequality~\eqref{eq::infsup_mcs-2} follows
from~\eqref{eq::contone}, \eqref{eq::conttwo}, and
\eqref{eq::contthree}.
\qqed
\end{proof}

\subsection{Pressure robust error analysis of MCS scheme}

In addition to the spaces $\Ureg$ and $\Qreg$, the {\it a priori}
error analysis will now also use a stress space with improved
regularity,
$ \Sreg := \Sigma^{\operatorname{sym}} \cap H^1(\mesh, \dd)$.  Note
that the integrals in the terms defining
$B(\sigma, u, \hat u, \omega, p; \cdot)$ are well-defined for
$\sigma \in \Sreg$, $(u, u_t, \omega) \in \Ureg$, and $p\in \Qreg$, so
$B(\cdot, \cdot)$ can be extended to such non-discrete arguments.

\begin{lemma}[Consistency of the MCS method]
  \label{lemma::mcs_cosistency}
  Assume that the exact solution $(\sigma, u,p)$ of
  \eqref{eq::mixedstressstokesweak} fulfills the regularity assumption
  $(u, u_t, \omega) \in \Ureg $ and
  $(\sigma, p) \in \Sreg \times \Qreg$, where $\omega = \curl(u)$. Let
  $(\sigma_h, u_h, \hat u_h, \omega_h, p_h) \in S_h$ be the solution of
  \eqref{eq::mcsstokes}
  and let $(\tau_h, v_h, \hat v_h, \eta_h, q_h) \in S_h$ be an
  arbitrary test function. Then 
\begin{align}\label{eq::mcs_cosistency}
  B(\sigma - \sigma_h,  u-u_h, u_t - \hat u_h, \omega - \omega_h, p - p_h; \tau_h, v_h, \hat{v}_h, \eta_h, q_h) = 0.
\end{align}
\end{lemma}
\begin{proof}
  Since $\sigma$ is symmetric we have that $\sigma :
  \kappa(\eta_h)=0$. Next, using the regularity assumptions,
  starting from \eqref{eq::mcsbblf}, we get 
\begin{align*}
  -\langle\div\sigma;\,
  &v_h, \hat{v}_h, \eta_h\rangle_{U_h} - (\div v_h, p_h)
    = (\sigma - pI :\nabla v_h)_\mesh - (\sigma_{nt} ,\cdot (v_h-\hat{v}_h)_t)_{\partial\mesh}\\
  & = -(\div(\sigma- pI),  v_h)_\mesh - (\sigma_{nt},  (v_h-\hat{v}_h)_t)_{\partial\mesh} +((\sigma - pI)n, v_h )_{\partial\mesh}\\
  & = -(\div(\sigma- pI),  v_h)_\mesh
    + (\sigma_{nt},  \hat{v}_h)_{\partial \mesh}
    - ((\sigma - pI)_{nn},  (v_h)_n)_{\partial \mesh}
  \\
  & = -(\div(\sigma- pI),  v_h)_\mesh
    + \sum_{F\in\facets} (\jump{\sigma}_{nt},  \hat{v}_h)_F
    - (\jump{(\sigma - pI)}_{nn},  (v_h)_n)_F
    \\
  & = -(\div(\sigma- pI),  v_h)_\mesh
    + \int_{\Gamma_N}(\sigma - pI)_{nn} (v_h)_n - \sigma_{nt} \hat{v}_h \ds\\
  & = -(\div(\sigma- pI),  v_h)
     = (f, v_h),
\end{align*}
where the boundary integral vanished using
\eqref{eq::mixedstressstokes-f} given on $\Gamma_N$.
Next, since
$\nu^{-1}\sigma=\eps(u)=\nabla u - \kappa(\omega)$ we have
\begin{align*}
  \nu^{-1} (\sigma, \tau_h)
  &+ \langle\div\tau_h; \,u, \hat{u}, \omega \rangle_{U_h} \\
  & =  \nu^{-1} (\sigma, \tau_h) -(\tau_h, \nabla u -\kappa(\omega))_\mesh
    + (\tau_{nt}, (u - u_t)_t)_{\partial \mesh} = 0.    
\end{align*}
The final remaining term in the bilinear form is also zero
since $(\div u, q_h) = 0$ as the exact solution is divergence free.
\qqed
\end{proof}

\begin{theorem}[Error estimate for the MCS method]
  \label{lemma::mcs_apr_error} Assume
  that the exact solution $(\sigma, u,p)$ of
  \eqref{eq::mixedstressstokesweak} fulfills the regularity assumption
  $(u, u_t, \omega) \in \Ureg $ and $(\sigma, p) \in \Sreg \times
  \Qreg$, where $\omega = \curl(u)$. Let $(u_h, \hat u_h, \omega_h)
  \in U_h$ and $(\sigma_h,p_h) \in \Sigma_h \times Q_h$ be the
  solution of \eqref{eq::mcsstokes}. Then we have  the pressure robust  error estimate 
  \begin{align}
    \nu^{-1}\|\sigma - \sigma_h\|_0
    + \|u-u_h, u_t - \hat{u}_h, \omega  -  \omega_h\|_{\eps} 
    \lesssim h \|u\|_{H^2(\mesh)} . \label{eq::mcs_apr_error_pres}
  \end{align}
Furthermore, the pressure error can be bounded by
\begin{align} \label{eq::mcs_apr_error_pres-2}
  \nu^{-1}\|p - p_h\|_0 \lesssim h \big(\|u\|_{H^2(\mesh)} + \nu^{-1}\|p\|_{H^1(\mesh)}\big).
\end{align}
\end{theorem}
\begin{proof}
  As in the proof of Theorem~\ref{thm::hdg_apr_error},  let 
$    E = (\sigma-\sigma_h, u-u_h, u_t - \hat{u}_h, \omega-\omega_h, p - p_h), $
$    E_h = (\IS\sigma-\sigma_h, \IV u-u_h, \IVh u_t - \hat{u}_h, \IW\omega-\omega_h, \IQ p-p_h),$
  and let the interpolation error be $\mathcal E = E - E_h$. 
  Now, using  Lemma \ref{lemma::infsup_mcs}, choose $s = (\tau_h, v_h, \hat{v}_h, \eta_h, q_h)$ such that
  \begin{align*}
      \|E_h\|_{S_h}
      & \lesssim \frac{B( E_h;s)}{\|s\|_{S_h}}.
  \end{align*}
  By the consistency of the MCS formulation \eqref{eq::mcs_cosistency} we  have
  \begin{align*}
    B(E_h;s) = B(E - \mathcal{E};s) = B(\mathcal{E};s).
  \end{align*}
  Hence, if we prove that
  \begin{align} \label{eq::toprove}
    B(\mathcal{E}; s) \lesssim \nu^{1/2}h \|u\|_{H^2(\mathcal{T})} \|s\|_{S_h},
  \end{align}
  then $\| E_h \|_{S_h} \lesssim \nu^{1/2}h \|u\|_{H^2(\mathcal{T})}$, which is enough to yield
  the stated pressure-independent estimate \eqref{eq::mcs_apr_error_pres}: indeed,
  letting $
  \bar E := (\sigma-\sigma_h, u-u_h, u_t - \hat{u}_h, \omega-\omega_h, 0)$, 
  $\bar E_h := (\IS\sigma-\sigma_h, \IV u-u_h, \IVh u_t - \hat{u}_h, \IW\omega-\omega_h, 0),$ and
   $\bar{\mathcal E} = \bar E - \bar E_h$,  we would then  have
   \begin{align}
     \label{eq:Ebaretc}
\| \bar E \|_{S_h} \le \| \bar{\mathcal E} \|_{S_h} + \| \bar E_h \|_{S_h} \le \| \bar{\mathcal E} \|_{S_h} + \| E_h \|_{S_h} \le \nu^{1/2} h \|u\|_{H^2(\mathcal{T})},
\end{align}
using the interpolation
estimates~\eqref{eq::approx_U}--\eqref{eq::approx_Sigma} to bound
$\| \bar{\mathcal E} \|_{S_h}$. Inequality~\eqref{eq:Ebaretc} obviously  implies~\eqref{eq::mcs_apr_error_pres}. Therefore we focus on
proving~\eqref{eq::toprove} and  proceed to separately
inspect each term forming its left hand side.

  Let
  $\mathcal{E}^j$ with $j \in \{\sigma, u, \hat u, \omega, p\}$ denote the
  corresponding components of the interpolation error.
  Then \eqref{eq::mcsbblf} implies 
  \begin{align*}
    \langle\div\tau_h; \mathcal{E}^u, \mathcal{E}^{\hat{u}}, \mathcal{E}^{\omega}\rangle
    & = (\tau_h, \kappa(\mathcal{E}^{\omega}) - \nabla \mathcal{E}^u)_\T
      + ((\tau_h)_{nt}, (\mathcal{E}^u - \mathcal{E}^{\hat{u}})_t)_{\d\T}.
  \end{align*}
  As $(\tau_h)_{nt}$ is constant on each facet,
  we can insert $\Pi^0$ in the last term, so several applications of the
  Cauchy-Schwarz inequality with  $h^{1/2}$ and $h^{-1/2}$ weights
  for the boundary terms yields
  \begin{align} \nonumber 
    \langle\div\tau_h;\, \mathcal{E}^u, \mathcal{E}^{\hat{u}}, \mathcal{E}^{\omega}\rangle
      & \lesssim \big( \|\tau_h\|_0
        +  h^{1/2}\|(\tau_h)_{nt}\|_{\partial \mathcal{T}}\big)(\|\mathcal{E}^u,\mathcal{E}^{\hat{u}}\|_{\nabla} + \|\kappa(\mathcal{E}^{\omega})\|_0)\\     \label{eq:Euuw}
    & \lesssim  \|\tau_h\|_0 h \|u\|_{H^2(\mathcal{T})},
    \end{align}
  where we used~\eqref{eq:easy-to-prove} again 
  and the
  interpolation estimate \eqref{eq::approx_U} in the last step.

  Next consider the symmetrically opposite term in $B$.
  Since $\nabla v_h\in P^0(\mesh)$ and $\mathcal{E}^{\sigma}$ is orthogonal to 
  facet-wise and element-wise constant functions (see \eqref{def::interp_Sigma_vol}--\eqref{def::interp_Sigma_edge}), we have 
  \begin{align*}
    -\langle\div \mathcal{E}^{\sigma}; v_h, \hat{v}_h, \eta_h\rangle_{U_h} & = (\mathcal{E}^\sigma, \nabla v_h - \kappa(\eta_h))_\T - (\mathcal{E}^\sigma_{nt}, (v_h - \hat{v}_h)_t)_{\partial \T} \\
    &=  -(\mathcal{E}^\sigma, (I-\Pi^0)\kappa(\eta_h))_\T - (\mathcal{E}^\sigma_{nt}, (I-\Pi^0)(v_h - \hat{v}_h)_t)_{\partial \T} \\
    &\lesssim  \|\mathcal{E}^\sigma\|_0 ~\|v_h, \hat{v}_h, \eta_h\|_{\eps} + h^{1/2}\|\mathcal{E}^\sigma_{nt}\|_{\partial \mathcal{T}} \| v_h,\hat v_h \|_\nabla.
  \end{align*}
  where on the right hand side of the last inequality, the first term is  obtained
using~\eqref{eq::normequiB} and Lemma~\ref{lemma::normequi}, while the second term is obtained using~\eqref{eq:1200}.
  Thus, the interpolation estimate \eqref{eq::approx_Sigma}
  and Lemma \ref{lem::korn_hdg} imply
  \begin{align}
    \label{eq:Esigma}
    \langle\div \mathcal{E}^{\sigma}; v_h, \hat{v}_h, \eta_h\rangle_{U_h} \lesssim \nu h \|u\|_{H^2(\mathcal{T})} \|v_h, \hat{v}_h, \eta_h\|_{\eps}.
  \end{align}
The  remaining terms are easy: by  Cauchy-Schwarz inequality, 
  \begin{align}
    \label{eq:Etausimple}
    \nu^{-1}(\mathcal{E}^\sigma, \tau_h) \lesssim  h \|u\|_{H^2(\mathcal{T})} \|\tau_h \|_0,
  \end{align}
 and by the definition of $I_W, I_Q$ and \eqref{eq:commut}, 
 \begin{align}
   \label{eq:EEElast}
    (\div \mathcal{E}^{\omega}, \div \eta_h ) &= 0,\quad 
    (\div \mathcal{E}^u, q_h) = 0, \quad \textrm{and} \quad
    (\mathcal{E}^p, \div v_h) = 0,
  \end{align}
where the last equation is due to $\div v_h\in P^0(\mesh)$.
Summing up~\eqref{eq:Euuw},
\eqref{eq:Esigma},
\eqref{eq:Etausimple}, and 
\eqref{eq:EEElast}, we prove~\eqref{eq::toprove}, and hence~\eqref{eq::mcs_apr_error_pres}.

The pressure error estimate \eqref{eq::mcs_apr_error_pres-2} follows
along the same lines as in the proof of
Theorem~\ref{thm::hdg_apr_error}.
\qqed
\end{proof}

\section{Numerical examples}\label{sec::numex}

In this last section we present a simple numerical example to provide
a practical illustration of the theoretical asymptotic convergence
rates as well as to compare the two new methods we presented.  Both
methods were implemented within the finite element library
NGSolve/Netgen (see \cite{netgen,schoeberl2014cpp11} and
\url{www.ngsolve.org}). Testfiles and our computational results
are available at \cite{philip_lederer_2023_7767775}.

The computational domain is given by $\Omega = (0,1)^3$ and the
velocity field is driven by the volume force determined by $f =
-\div \sigma  + \nabla p$ with the exact solution given by
\begin{align*}
  \sigma &= \nu\eps(\curl(\psi,\psi,\psi)), \quad \textrm{and} \quad  p := x^5 + y^5 +z^5 - \frac{1}{2}.                              
\end{align*}
Here $\psi :=x^2(x-1)^2y^2(y-1)^2z^2(z-1)^2$ defines a given potential
and we choose the viscosity $\nu = 10^{-4}$.
While this would lend itself to homogenous Dirichlet conditions being prescribed on the
whole boundary, as we assume $|\Gamma_N|>0$ throughout the paper, we instead opt
to impose non-homogenous Neumann conditions on $\Gamma_N:=\{0\}\times(0,1)\times(0,1)$
and homogenous Dirichlet conditions only on $\Gamma_D:=\partial\Omega\setminus\Gamma_N$.
Note that this requires the additional source terms $\int_{\Gamma_N}(\sigma_{nn} - p) (v_h)_n\ds$
and $\int_{\Gamma_N}\sigma_{nt}\hat{v}_h\ds$ to be provided as data for the methods.

\bigskip\paragraph{\textbf{Convergence}} An initial, relatively coarse
mesh was generated and then refined multiple times.
With the larger problem size on finer meshes in mind, we used a GMRes Krylov space solver 
preconditioned by an auxiliary space method using a lowest order conforming $H^1$ space
(see e.g., \cite{Fu2021}, and for details specific to the MCS case, see~\cite{stokesAux})
with relative tolerance of $10^{-14}$.
Errors measured in different norms and their estimated
order of convergence (eoc) are listed in Table~\ref{tab::conv_hdg}
for the HDG method and Table~\ref{tab::conv_mcs} for the MCS method. For the
HDG method we chose the stabilization parameter $\alpha=6$. As
predicted by the analysis from Theorem~\ref{thm::hdg_apr_error} and
Lemma~\ref{lemma::mcs_apr_error}, the velocity error measured in the
seminorm $\| \varepsilon(u - u_h)\|_0$,  the $L^2$-norm of the
vorticity, and the pressure errors converge at optimal linear order.
Furthermore, for the MCS method, we also observe optimal convergence
for  the stress error. In addition,  we also plotted the $L^2$-norm error
of the velocity. From an Aubin-Nitsche argument one may 
expect a higher order of convergence whenever the dual problem shows
enough regularity~\cite{brezzi2012mixed, mcsI}. Not surprisingly therefore,
 we observe
 quadratic convergence for the $L^2$-norm of the velocity error  for
both methods.

\bigskip\paragraph{\textbf{Condition numbers}} For both HDG and MCS
method, after static condensation within the $(u_h,\hat{u}_h,
\omega_h)$- or $(\sigma_h, u_h, \hat{u}_h, \omega_h)$-block of the
finite element matrix respectively, we obtain a symmetric and positive
definite diagonal block, which we  simply refer to here
as the ``$A$''-blocks of the respective methods.
(Of course, due to the incompressibility constraint, the entire system is
still of saddle point structure.) Both the $A$ blocks have the same
non-zero structure and are expected to have condition number
$\mathcal{O}(h^{-2}),$ but they discretize slightly different operators, namely $\eps$
for the HDG method, and $\dev(\eps)$ for the MCS method. As $\eps(u) =
\dev(\eps(u)) + \frac{1}{3}\div (u) \id$ and the true solution is
divergence-free, adding the (consistent) term
$\frac{1}{3}\div u_h \div v_h $ to the MCS bilinear form yields an $A$
block that is directly comparable to the one of the HDG method. In
Figure \ref{fig::kappas} we show approximate condition numbers
($\operatorname{cond}$) of said $A$ blocks for some of the meshes used
in the previous computations and different values of $\alpha$ in
$a^{\hdg}$. We see that in addition to the MCS method not being
dependent on any stabilization parameter in the first place, there
appears to be no possible choice of $\alpha$ that would make the HDG method's $A$ block  better conditioned than that of the MCS method.

\begin{table}
\begin{tabular}{
  r|
  c !{\!(\!\!\!}c!{\!\!\!)}
  c !{\!(\!\!\!}c!{\!\!\!)}
  c !{\!(\!\!\!}c!{\!\!\!)}
  c !{\!(\!\!\!}c!{\!\!\!)}
  }
  $|\mathcal{T}|$
  & $\| \eps( u- u_h)\|_0$ & \footnotesize eoc
  &$\| u - u_h\|_0$ & \footnotesize eoc
  &$\| \omega - \omega_h\|_0$ & \footnotesize eoc
  &$\| p - p_h\|_0$ & \footnotesize eoc \\
\midrule
  63 & \num{0.0022173599660678116}&--&\num{0.00019090886564967198}&--& \num{0.003248492075177081}&--& \num{0.21231118857781317}&--\\
  504 & \num{0.0017074225194699984}&\numeoc{0.3770228837872961}& \num{8.376278540597054e-05}&\numeoc{1.1885027790190954}& \num{0.002256073294951415}&\numeoc{0.5259562471321205}& \num{0.11648126812663936}&\numeoc{0.86608243447335}\\
  4032 & \num{0.0009345156759792808}&\numeoc{0.8695293427557381}& \num{2.380571599859018e-05}&\numeoc{1.8150013966029703}& \num{0.0011955926546911224}&\numeoc{0.9160879998106013}& \num{0.06052251541685397}&\numeoc{0.9445541125651082}\\
  32256 & \num{0.0005289036827525488}&\numeoc{0.8212138436886784}& \num{7.993236356447451e-06}&\numeoc{1.5744563662963298}& \num{0.000662772094794181}&\numeoc{0.8511411726889326}& \num{0.031013284227886167}&\numeoc{0.9645856378425158}\\
  258048 & \num{0.0002778705621991993}&\numeoc{0.92859201837252}& \num{2.3188796658537836e-06}&\numeoc{1.7853517938904278}& \num{0.00034623221227850134}&\numeoc{0.9367729064989447}& \num{0.015615445640843056}&\numeoc{0.989912569526987}\\
  2064384 & \num{0.0001419097998698605}&\numeoc{0.9694387814379389}& \num{6.307983284046898e-07}&\numeoc{1.8781772136761787}& \num{0.00017671377729121411}&\numeoc{0.9703254320614225}& \num{0.007821817698817814}&\numeoc{0.9973979252107024}
\end{tabular}
\caption{Errors and estimated order of convergence (eoc) for the HDG method.} \label{tab::conv_hdg}
\end{table}

\begin{table}
\begin{adjustbox}{center}
%
\begin{tabular}{
  r|
  c !{\!(\!\!\!}c!{\!\!\!)}
  c !{\!(\!\!\!}c!{\!\!\!)}
  c !{\!(\!\!\!}c!{\!\!\!)}
  c !{\!(\!\!\!}c!{\!\!\!)}
  c !{\!(\!\!\!}c!{\!\!\!)}
  }
  $|\mathcal{T}|$
  & $\| \eps(u - u_h)\|_0$ & \footnotesize eoc
  &$\| u - u_h\|_0$ & \footnotesize eoc
  &$\| \sigma - \sigma_h\|_0$ & \footnotesize eoc
  &$\| \omega - \omega_h\|_0$ & \footnotesize eoc
  &$\| p - p_h\|_0$ & \footnotesize eoc \\
\midrule
  63 &      \num{0.002611775142138412}&--&\num{0.00020733664426244024}&--& \num{3.9747121660207704e-07}&--& \num{0.0031826048425160694}&--& \num{0.21231011812347372}&--\\
  504 &     \num{0.0019139100723314741}&\numeoc{0.4485076505865729}& \num{0.00010109800356526557}&\numeoc{1.036220610134003}& \num{2.857739643935184e-07}&\numeoc{0.47597590469810797}& \num{0.0021792691255353115}&\numeoc{0.5463636697979489}& \num{0.11647847239863739}&\numeoc{0.8661097878093227}\\
  4032 &    \num{0.0009983902759762942}&\numeoc{0.93884725646567}& \num{2.5330657117148225e-05}&\numeoc{1.9967980994441659}& \num{1.4611912343308901e-07}&\numeoc{0.9677294807056731}& \num{0.0011187547199830708}&\numeoc{0.961950601827103}& \num{0.06052098143722692}&\numeoc{0.9445560517079415}\\
  32256 &   \num{0.0006047847478495912}&\numeoc{0.7231821263574093}& \num{7.827067631521843e-06}&\numeoc{1.6943406873468894}& \num{7.889606506302894e-08}&\numeoc{0.889119751319206}& \num{0.000606159659253402}&\numeoc{0.8841240219445474}& \num{0.031012329390826925}&\numeoc{0.9645934897855617}\\
  258048 &  \num{0.0003057836087706834}&\numeoc{0.9839106819675295}& \num{2.0494118803867253e-06}&\numeoc{1.9332619518503469}& \num{4.010054944149707e-08}&\numeoc{0.9763313436569098}& \num{0.0003067428221888684}&\numeoc{0.9826682569249834}& \num{0.015614912413415945}&\numeoc{0.9899174163226863}\\
  2064384 & \num{0.00015037879556418213}&\numeoc{1.0239099230077853}& \num{5.155134077348788e-07}&\numeoc{1.991128103467681}& \num{2.0076185794699295e-08}&\numeoc{0.9981368013720411}& \num{0.0001528979283428527}&\numeoc{1.004460725733168}& \num{0.00782153597297643}&\numeoc{0.9974006238874558}
\end{tabular}

\end{adjustbox}
\caption{Errors and estimated order of convergence (eoc) for the MCS
method.} \label{tab::conv_mcs}
\end{table}

\begin{figure}
  \resizebox{.8\textwidth}{!} {
    \input{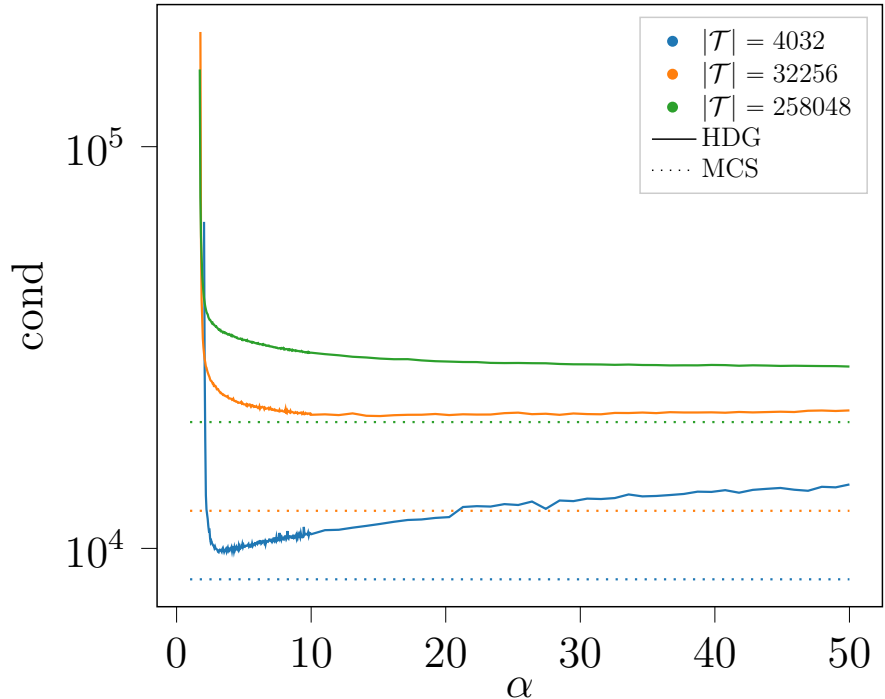}
    }
  \caption{Approximate condition numbers of the corresponding $A$
    blocks of the HDG (solid lines) and the MCS (dotted lines in the
    same color) method on different meshes. Different values of
    $\alpha$ on the x axis and approximate condition number ($\operatorname{cond}$) on the y axis. }
  \label{fig::kappas}
\end{figure}


\section*{Acknowledgements}

\begin{footnotesize}
The authors acknowledge support from the Austrian Science Fund (FWF)
through the research program ``Taming complexity in partial
differential systems'' (F65) - project ``Automated discretization in
multiphysics'' (P10), the research program W1245, and the 
NSF grants 1912779, 2136228.  
\end{footnotesize}

\bibliographystyle{siam}
\bibliography{literature}

\begin{thebibliography}{10}

\bibitem{MR840802}
{\sc D.~N. Arnold, F.~Brezzi, and J.~Douglas, Jr.}, {\em P{EERS}: a new mixed finite element for plane elasticity}, Japan J. Appl. Math., 1 (1984), pp.~347--367.

\bibitem{MR2336264}
{\sc D.~N. Arnold, R.~S. Falk, and R.~Winther}, {\em Mixed finite element methods for linear elasticity with weakly imposed symmetry}, Math. Comp., 76 (2007), pp.~1699--1723.

\bibitem{MR2449101}
{\sc D.~Boffi, F.~Brezzi, and M.~Fortin}, {\em Reduced symmetry elements in linear elasticity}, Commun. Pure Appl. Anal., 8 (2009), pp.~95--121.

\bibitem{brezzi2012mixed}
\leavevmode\vrule height 2pt depth -1.6pt width 23pt, {\em Mixed Finite Element Methods and Applications}, Springer Science \& Business Media, 2013.

\bibitem{brenner_korn}
{\sc S.~C. Brenner}, {\em Korn's inequalities for piecewise {$H^1$} vector fields}, Math. Comp., 73 (2004), pp.~1067--1087.

\bibitem{MR2629995}
{\sc B.~Cockburn, J.~Gopalakrishnan, and J.~Guzm\'{a}n}, {\em A new elasticity element made for enforcing weak stress symmetry}, Math. Comp., 79 (2010), pp.~1331--1349.

\bibitem{cockburn2009unified}
{\sc B.~Cockburn, J.~Gopalakrishnan, and R.~Lazarov}, {\em Unified hybridization of discontinuous {Galerkin}, mixed, and continuous {Galerkin} methods for second order elliptic problems}, SIAM Journal on Numerical Analysis, 47 (2009), pp.~1319--1365.

\bibitem{cockburn2005locally}
{\sc B.~Cockburn, G.~Kanschat, and D.~Sch{\"o}tzau}, {\em A locally conservative {LDG} method for the incompressible {Navier}-{Stokes} equations}, Mathematics of Computation, 74 (2005), pp.~1067--1095.

\bibitem{cockburn2007note}
\leavevmode\vrule height 2pt depth -1.6pt width 23pt, {\em A note on discontinuous {Galerkin} divergence-free solutions of the {Navier}--{Stokes} equations}, Journal of Scientific Computing, 31 (2007), pp.~61--73.

\bibitem{CrouzRavia73}
{\sc M.~Crouzeix and P.-A. Raviart}, {\em Conforming and nonconforming finite element methods for solving the stationary {S}tokes equations~{I}.}, RAIRO (Revue Fran{\c{c}}aise d'Atomatique, Informatique et Recherche Op{\'{e}}rationnelle), Analyse Num{\'{e}}rique, R-3 ($7^{\mathrm{e}}$ ann{\'{e}}e) (1973), pp.~33--76.

\bibitem{Falk91}
{\sc R.~S. Falk}, {\em Nonconforming finite element methods for the equations of linear elasticity}, Math. Comp., 57 (1991), pp.~529--550.

\bibitem{Farhloulcanadian}
{\sc M.~Farhloul}, {\em Mixed and nonconforming finite element methods for the {S}tokes problem}, Canad. Appl. Math. Quart., 3 (1995), pp.~399--418.

\bibitem{MR1231323}
{\sc M.~Farhloul and M.~Fortin}, {\em A new mixed finite element for the {S}tokes and elasticity problems}, SIAM J. Numer. Anal., 30 (1993), pp.~971--990.

\bibitem{MR1464150}
{\sc M.~Farhloul and M.~Fortin}, {\em Dual hybrid methods for the elasticity and the {S}tokes problems: a unified approach}, Numer. Math., 76 (1997), pp.~419--440.

\bibitem{MR1934446}
{\sc M.~Farhloul and M.~Fortin}, {\em Review and complements on mixed-hybrid finite element methods for fluid flows}, in Proceedings of the 9th {I}nternational {C}ongress on {C}omputational and {A}pplied {M}athematics ({L}euven, 2000), vol.~140, 2002, pp.~301--313.

\bibitem{Fu2021}
{\sc G.~Fu}, {\em Uniform auxiliary space preconditioning for {HDG} methods for elliptic operators with a parameter dependent low order term}, {SIAM} Journal on Scientific Computing, 43 (2021), pp.~A3912--A3937.

\bibitem{MR3941890}
{\sc G.~Fu, Y.~Jin, and W.~Qiu}, {\em Parameter-free superconvergent {$H({\rm div})$}-conforming {HDG} methods for the {B}rinkman equations}, IMA J. Numer. Anal., 39 (2019), pp.~957--982.

\bibitem{girault2012finite}
{\sc V.~Girault and P.-A. Raviart}, {\em Finite element methods for Navier-Stokes equations: theory and algorithms}, vol.~5, Springer Science \& Business Media, 2012.

\bibitem{GopalGuzma12}
{\sc J.~Gopalakrishnan and J.~Guzm{\'a}n}, {\em A second elasticity element using the matrix bubble}, {IMA} J. Numer. Anal., 32 (2012), pp.~352--372.

\bibitem{mcsI}
{\sc J.~Gopalakrishnan, P.~L. Lederer, and J.~Sch\"{o}berl}, {\em A mass conserving mixed stress formulation for the {Stokes} equations}, {IMA} J. Numer. Anal., 40 (2019), pp.~1838--1874.

\bibitem{mcsII}
{\sc J.~Gopalakrishnan, P.~L. Lederer, and J.~Sch\"{o}berl}, {\em A mass conserving mixed stress formulation for {Stokes} flow with weakly imposed stress symmetry}, SIAM J. Numer. Anal., 58 (2020), pp.~706--732.

\bibitem{stokesAux}
{\sc L.~Kogler, P.~L. Lederer, and J.~Sch\"{o}berl}, {\em A conforming auxiliary space preconditioner for the mass conserving stress-yielding method}, 2022.

\bibitem{LandaLifsh59}
{\sc L.~D. Landau and E.~M. Lifshitz}, {\em Fluid mechanics}, Translated from the Russian by J. B. Sykes and W. H. Reid. Course of Theoretical Physics, Vol. 6, Pergamon Press, London, 1959.

\bibitem{philip_lederer_2023_7767775}
{\sc P.~Lederer, L.~Kogler, J.~Gopalakrishnan, and J.~Schöberl}, {\em {Computational results and python files for the work "Divergence-conforming velocity and vorticity approximations for incompressible fluids obtained with minimal facet coupling", 10.5281/zenodo.7767775 }}, Mar. 2023.

\bibitem{lederer2019mass}
{\sc P.~L. Lederer}, {\em A Mass Conserving Mixed Stress Formulation for Incompressible Flows}, PhD thesis, Technical University of Vienna, 2019.

\bibitem{ledlehrschoe2017relaxedpartI}
{\sc P.~L. Lederer, C.~Lehrenfeld, and J.~Sch\"{o}berl}, {\em Hybrid discontinuous {G}alerkin methods with relaxed {$H({{div}})$}-conformity for incompressible flows. {P}art {I}}, SIAM Journal on Numerical Analysis, 56 (2018), pp.~2070--2094.

\bibitem{ledlehrschoe2018relaxedpartII}
\leavevmode\vrule height 2pt depth -1.6pt width 23pt, {\em Hybrid discontinuous {G}alerkin methods with relaxed {$H({{div} })$}-conformity for incompressible flows. {P}art {II}}, ESAIM. Mathematical Modelling and Numerical Analysis, 53 (2019), pp.~503--522.

\bibitem{2016arXiv160903701L}
{\sc P.~L. Lederer, A.~Linke, C.~Merdon, and J.~Sch\"oberl}, {\em Divergence-free {R}econstruction {O}perators for {P}ressure-{R}obust {S}tokes {D}iscretizations with {C}ontinuous {P}ressure {F}inite {E}lements}, SIAM J. Numer. Anal., 55 (2017), pp.~1291--1314.

\bibitem{LS_CMAME_2016}
{\sc C.~Lehrenfeld and J.~Sch\"{o}berl}, {\em High order exactly divergence-free {Hybrid Discontinuous Galerkin} methods for unsteady incompressible flows}, Computer Methods in Applied Mechanics and Engineering, 307 (2016), pp.~339 -- 361.

\bibitem{linke2014role}
{\sc A.~Linke}, {\em On the role of the {Helmholtz} decomposition in mixed methods for incompressible flows and a new variational crime}, Computer Methods in Applied Mechanics and Engineering, 268 (2014), pp.~782--800.

\bibitem{MardaWinth06}
{\sc K.-A. Mardal and R.~Winther}, {\em An observation on {K}orn's inequality for nonconforming finite element methods}, Math. Comp., 75 (2006), pp.~1--6.

\bibitem{netgen}
{\sc J.~Sch{\"o}berl}, {\em {NETGEN An advancing front 2D/3D-mesh generator based on abstract rules}}, Computing and Visualization in Science, 1 (1997), pp.~41--52.

\bibitem{schoeberl2014cpp11}
{\sc J.~Sch\"oberl}, {\em C++11 implementation of finite elements in {NGSolve}}, Tech. Rep. ASC-2014-30, Institute for Analysis and Scientific Computing, September 2014.

\bibitem{Stenberg1988}
{\sc R.~Stenberg}, {\em A family of mixed finite elements for the elasticity problem}, Numerische Mathematik, 53 (1988), pp.~513--538.

\bibitem{TaiWinth06}
{\sc X.-C. Tai and R.~Winther}, {\em A discrete de {R}ham complex with enhanced smoothness}, Calcolo, 43 (2006), pp.~287--306.

\end{thebibliography}

\end{document}